\documentclass[11pt]{amsart}
 \usepackage{amsmath,amssymb,enumerate}
 \usepackage[all]{xy}
 \newtheorem{prop}{Proposition}[section]
 \newtheorem{coro}[prop]{Corollary}

 \newtheorem{lem}[prop]{Lemma}
 \newtheorem{rem}[prop]{Remark}

  \newtheorem{defi}[prop]{Definition}
 \newtheorem{theo}[prop]{Theorem}

 \newcommand{\R}{\mathbb R}
 \newcommand{\Q}{\mathbb Q}
 \newcommand{\C}{\mathbb C}
 \newcommand{\N}{\mathbb N}

 \newcommand{\e}{\varepsilon}
 \newcommand{\f}{\varphi}
 
 \newcommand{\p}{\psi}

 \newcommand \al {\alpha}

 \numberwithin{equation}{section}

\begin{document}

  \title[Parabolic complex Monge-Amp\`ere equations]{Weak solutions to degenerate complex Monge-Amp\`ere Flows II}

\setcounter{tocdepth}{1}

  \author{ Philippe Eyssidieux, Vincent Guedj, Ahmed Zeriahi} 

\address{Universit\'e Joseph Fourier et Institut Universitaire de France, France}

\email{Philippe.Eyssidieux@ujf-grenoble.fr }

\address{Institut de Math\'ematiques de Toulouse et Institut Universitaire de France, France}

\email{vincent.guedj@math.univ-toulouse.fr}

\address{Institut de Math\'ematiques de Toulouse,   France}

\email{zeriahi@math.univ-toulouse.fr}

 \date{\today}

\begin{abstract}  
Studying the (long-term) behavior of the K\"ahler-Ricci flow on mildly singular varieties, one is naturally lead to study weak solutions of degenerate parabolic complex Monge-Amp\`ere equations.

  The purpose of this article, the second of a series on this subject, is to develop a viscosity theory for
degenerate complex Monge-Amp\`ere flows on compact K\"ahler manifolds.  Our general theory
allows in particular to define and study the (normalized) K\"ahler-Ricci flow on varieties with canonical singularities, generalizing results of Song and Tian.
\end{abstract}

 \maketitle

\tableofcontents

\newpage

\section*{Introduction}

 The study of the (long-term) behavior of the K\"ahler-Ricci flow on mildly singular varieties in relation to the Minimal Model Program was undertaken by J. Song and G. Tian \cite{ST1,ST2} and requires 
a theory of weak solutions for certain  degenerate parabolic complex Monge-Amp\`ere equations modelled on:
\begin{equation}\label{krf}
 \frac{\partial \phi}{\partial t}+\phi= \log \frac{(dd^c\phi)^n}{V}
\end{equation}
where $V$ is volume form and $\phi$ a $t$-dependant K\"ahler potential on a compact K\"ahler manifold. The approach in \cite{ST2} is to regularize the equation
and take limits of the solutions of the regularized equation with uniform
higher order estimates. But as far as the existence and uniqueness statements in \cite{ST2} are concerned, we believe that 
a zeroth order approach would be both simpler and more efficient.

There is a well established pluripotential 
theory of weak solutions to elliptic complex Monge-Amp\`ere equations,
following the pionnering work of Bedford and Taylor \cite{BT76,BT82} in the local case (domains in $\C^n$). 
A complementary viscosity approach has been developed only recently in \cite{HL09,EGZ11,W12, EGZ13} both in the local and the global case (compact K\"ahler manifolds).

Suprisingly no similar theory has ever been developed on the parabolic side.  The most significant reference for a parabolic flow of plurisubharmonic functions on a pseudoconvex domains
 is \cite{Gav} but the  flow studied there takes  the form 
\begin{equation}
 \label{gavflow}  \frac{\partial \phi}{\partial t}= ((dd^c\phi)^n )^{1/n} 
\end{equation}
which does not make sense in the global case. 
The purpose of this article, the second of a series on this subject, is to develop a viscosity theory for
degenerate complex Monge-Amp\`ere flows of the form 
(\ref{krfglobalal}).

   This article focuses on solving this problem on compact K\"ahler manifolds, while its companion \cite{EGZ14} is concerned with the local case (domains in $\C^n$).
More precisely  we study here the complex degenerate parabolic complex Monge-Amp\` ere flows
\begin{equation} \label{krfglobalal}
 e^{\dot{\f}_t + F (t,x,\f) } \mu (t,x) - (\omega_t +dd^c \f_t)^n = 0, 
\end{equation}
where
  \begin{itemize}
   \item $T \in ]0,+\infty]$;
  \item $\omega=\omega(t,x)$ is a continuous family of semi-positive $(1,1)$-forms on $X$, 
  \item $F (t,z,r)$ is continuous in $[0,T[ \times X \times \R$ and non decreasing in $r$,
  \item $\mu (t,z)\geq 0$ is a bounded continuous volume form on $X$,
   \item $\f : X_T :=  [0,T[ \times X \rightarrow \R$ is the unknown function, with $\f_t: = \f (t,\cdot)$. 
   \end{itemize}
Our plan is to adapt the viscosity approach developped by P.L.~Lions and al.
(see \cite{IL90,CIL92}) to the complex case, using the elliptic side of the theory which was developped in \cite{EGZ11}. It should be noted that the
method used in \cite{ST2} is a version of the classical PDE method of vanishing viscosity which was superseded by the theory of viscosity solutions.

\smallskip

We develop the appropriate definitions of (viscosity) subsolution, supersolution and solution
in the {\it first section}, and connect these to weak solutions of the K\"ahler-Ricci flow (normalized or not).

As is often the case in the viscosity theory, one of our main technical tools is the global comparison principle.
We actually establish several comparison principles in the {\it second section}, in particular
the following:

\medskip

\noindent{\bf{Theorem A.}} 
{\em  
Assume $t \mapsto \omega_t$ is non-decreasing or more generally regular in the sense of Definition \ref{defi:oregular}.
If $\f$ (resp. $\p$) is a bounded subsolution (resp. supersolution) 
of the above degenerate parabolic equation
then
$$
\max_{X_T} (\f - \p) \leq \max_{x \in X} (\f (0,x)- \p(0,x))_+,
$$

with the notation $a_+= \max(a,0)$, given $a$ a real number.
}

\smallskip We do not reproduce here the rather technical Definition \ref{defi:oregular} and refer the reader to section 2 instead. It is enough to record here that 
the condition is 
satisfied in all the situations arising from the  K\"ahler-Ricci flow with singularities.

\medskip

In the {\it third section} we specialise to the complex Monge-Amp\`ere flows arising in the study of 
the (normalized) K\"ahler-Ricci flow on mildly singular varieties, assuming
$F(t,x,\f)=\al \f$,
and 
$$
\mu(x,t)=e^{u(x)} f(x,t) dV(x),
$$
where $f>0$ is a positive continuous density and $u$ is  quasi-plurisubharmonic function that is
  {\it exponentially continuous} (i.e. such that $e^u$ is continuous).

We construct barriers at each point of the parabolic boundary and use the Perron method to eventually show the existence of a viscosity solution to the Cauchy problem:

\medskip

\noindent{\bf{Theorem B.}} 
{\em  
 Let  $\f_0$ be a continuous $\omega_0$-plurisubharmonic function on $X$
and assume $F,\mu$ are as above.
The Cauchy problem for the parabolic complex Monge-Amp\`ere equation  
with initial data $\f_0$ admits a unique viscosity solution $\f (t,x)$;
it is the upper envelope of all subsolutions. 
}

\medskip

We describe applications to the K\"ahler-Ricci flow on varieties with a definite first Chern class in the
{\it fourth section}, showing in particular a generalization of Cao's theorem \cite{Cao85}:

\medskip

\noindent{\bf{Theorem C.}} 
{\em  
Let $Y$ be a $\Q$-Calabi-Yau  variety and $S_0$ a positive closed current with continuous potentials 
representing a K\"ahler class $\al \in H^{1,1}(Y,\R)$.
The K\"ahler-Ricci flow 
$$
\frac{\partial \omega_t}{\partial t}=-\rm{Ric}(\omega_t)
$$
can be uniquely run  from $S_0$ and converges, as $t \rightarrow +\infty$, towards the unique 
Ricci flat K\"ahler-Einstein current $S_{KE}$ in $\al$.
}

\medskip

We similarly handle the case of canonical models:

\medskip

\noindent{\bf{Theorem D.}} 
{\em  Let $Y$ be a canonical model, i.e. a general type
projective algebraic variety with only canonical singularities 
such that $K_Y$ is nef and big and $S_0$ a positive closed current with continuous potential 
representing a K\"ahler class $\al \in H^{1,1}(Y,\R)$.
The normalized K\"ahler-Ricci flow 
$$
\frac{\partial \omega_t}{\partial t}=-\rm{Ric}(\omega_t)-\omega_t
$$
can be uniquely run from $S_0$ and exists for all time. Moreover $\omega_t$ has continuous potentials on 
$\R^+ \times Y$ and  converges, as $t \rightarrow +\infty$, towards the unique singular K\"ahler-Einstein metric $S_{KE}$ on $Y$.

}

\medskip

The convergence is here uniform at the level of potentials. The existence of $S_{KE}$
is due to \cite{EGZ09}, while the continuity of its potentials follows from the elliptic viscosity
approach of \cite{EGZ11}. 

We also show that the weak K\"ahler-Ricci flows considered
by Song-Tian \cite{ST2} (when the measure $\mu$ is sufficiently regular) coincide with ours,
this yields in particular the global continuity of the corresponding potentials which was not established in \cite{ST2}.

\medskip

We conclude by proposing a (discontinuous) viscosity approach to understanding the behavior
of the K\"ahler-Ricci flow over the flips. This requires to extend our results allowing for
discontinuous densities, a promising line of research for the future. We plan to come back to that question
in a forthcoming work. 

\medskip
We learnt of the possibility to use  viscosity solutions in the present context via a hint in \cite{CLN} where  no attempt to fully justify this technique was made.



\section{Complex Monge-Amp\`ere flows on compact manifolds}

\subsection{Geometrical Background for Complex Monge-Amp\`ere flows}
Let $X$ be a $n$-dimensionnal compact complex manifold and $n=\dim_{\C}(X)$.

The sheaf  $\mathcal{Z}^{1,1}_X$ of closed $(1,1)$-forms with continuous potential 
is, by definition,  the quotient  sheaf $\mathcal{Z}^{1,1}_{X} :=\mathcal{C}^0_{X}/\mathcal{PH}_{X}$ 
of the sheaf $\mathcal{C}^0_{X}$  of real valued continuous functions on $X$ by its subsheaf of pluriharmonic functions. Given a section of $\mathcal{Z}^{1,1}_X$
represented by a cocycle $(\phi_{\beta})_{\beta\in B}$ where $\phi_{\beta} \in C^0(U_{\beta}, \R)$ and  $\mathfrak{U} =(U_\beta)_{\beta\in B}$ is covering 
of $X$,  the currents $dd^c\phi_{\beta}$ defined on each $U_\beta$ glue into a closed current of bidegree $(1,1)$ on $X$. 
Global sections of $\mathcal{Z}^{1,1}_X$ form the background geometry in the study of degenerate complex Monge-Amp\`ere equations in the global case \cite{EGZ09}.

It is straightforward to formulate a parabolic analog. Let $T$ be positive real number and consider the manifold with boundary $X_T:= [0,T[ \times  X$ 
and denote by $\mathcal{C}^0_{X_T}$ the sheaf of continuous functions 
on $X_T$.  Denote by $\mathcal{PH}_{X_T/[0,T[} \subset \mathcal{C}^0_{X_T}$
the sheaf of continuous  real valued local functions whose restriction to each 
$X_t:=  \{t\} \times X \buildrel{i_t}\over{\hookrightarrow} X_T$ is a  pluriharmonic real valued function.

Say a germ of real valued  function  on $X_T$ is of class $C^{1,2}$ if it is of class $C^1$
admitting continuous second order partial derivatives in the $X$ direction. 

\begin{defi}\label{defi:backgd}
 A   family of closed real $(1,1)$-forms with continuous local potentials $\omega=(\omega_t)_{t\in [0,T[}$ is a global section of the sheaf 
$\mathcal{Z}^{1,1}_{X_T/[0,T[} :=\mathcal{C}^0_{X_T}/\mathcal{PH}_{X_T/[0,T[}$.

 A  continuous family of closed real $(1,1)$-forms $\omega=(\omega_t)_{t\in [0,T[}$ is a global section of the sheaf 
$C^0\mathcal{Z}^{1,1}_{X_T/[0,T[} :=\mathcal{C}^{1,2}_{X_T}/\mathcal{PH}_{X_T/[0,T[}$. 
\end{defi}

It is straighforward to see that there is a covering $\mathfrak{U} =(U_\beta)_{\beta\in B}$
of $X$ such that,  for every $\omega$ a global section of the sheaf $\mathcal{Z}^{1,1}_{X_T/\R} :=\mathcal{C}^0_{X_T}/\mathcal{PH}_{X_T/\R}$, 
  $\omega|_{[0,T[\times U_\beta}$ is represented by $\Phi_{\beta}\in C^0([0,T[ \times U_\beta, \R)$
such that $\Phi_{\beta\beta'}=\Phi_{\beta}-\Phi_{\beta'} \in C^0(U_{\beta \beta'})$ satisfies
$\partial\bar\partial \Phi_{\beta\beta'}=0$ and conversely such a cochain $(\Phi_{\beta})_{\beta\in B}$
defines a global section of $\mathcal{Z}^{1,1}_{X_T/[0,T[}$. The covering $\mathfrak{U}$ will be fixed throughout the article for technical reasons but  our results
will not depend on this choice. 

We have a natural map $\mathcal{Z}^{1,1}_{X_T/[0,T[} \to (i_t)_*\mathcal{Z}^{1,1}_X$ hence 
for every $t\in [0,T[$,  $\omega$ defines a closed real $(1,1)$-form with continuous potentials $\omega_t$  
by the prescription:
$$
\left(
\omega\mapsto \omega_t:=dd^c\Phi_{\beta}|_{\{t\} \times U_{\beta}}, \quad  H^0(X_T, \mathcal{Z}^{1,1}_{X_T/[0,T[}) \to H^{0}(X,\mathcal{Z}^{1,1}_X)
\right)
$$
and, taking the Bott Chern cohomology class $\{ \omega_t\}$  of $\omega_t$, we get a map
$$
\{ \relbar _t\}: H^0(X_T, \mathcal{Z}^{1,1}_{X_T/[0,T[}) \to H^{1,1}_{BC}(X,\R) \simeq H^1(X,\mathcal{PH}_X)
$$
such that $t\mapsto \{\omega_t\}$
is a continuous $H^{1,1}_{BC}(X,\R)$-valued function. 
The resulting  map $ H^0(X_T, \mathcal{Z}^{1,1}_{X_T/[0,T[})\to C^0([0,T[, H^{1,1}_{BC}(X,\R))$
is surjective. On the other hand, $ H^0(X_T, C^0\mathcal{Z}^{1,1}_{X_T/[0,T[})$ maps onto $  C^1([0,T[, H^{1,1}_{BC}(X,\R))$

Let us remark that the previous definitions make sense for normal complex spaces. However, for the formulation of the flows to be given in the next paragraph, it is necessary  to assume smoothness. 

\subsection{Complex Monge-Amp\`ere flows}

\begin{defi} \label{cmafdef} The complex Monge-Amp\`ere flow associated to $(\omega, \mu, F)$ where:
 \begin{itemize}
  \item $\omega \in  H^0(X_T, C^0\mathcal{Z}^{1,1}_{X_T/[0,T[})$ is  a continuous family of closed  real $(1,1)-$forms on $X$ in the sense of definition \ref{defi:backgd};
\item $0\leq \mu (t,x) \in C^0(X, \Omega^{n,n}_{X_T/[0,T[} )$ is a continuous family of volume forms on $X$, 
  \item $F : [0,T[ \times X \times \R \longrightarrow \R$ is continuous and non decreasing in the last variable, 
 \end{itemize}
is the following parabolic equation:
$$
 (\omega +dd^c\phi)^n= e^{\frac{\partial \phi}{\partial t} + F (t,x,\phi)} \mu,
\leqno (CMAF)_{X,\omega, \mu, F}  
$$
\end{defi}
Here: 
  $\phi : X_T  \longrightarrow \R$ is the unknown function.





A classical solution of a complex Monge-Amp\`ere flow is a function of class $C^{1,2}$  satisfying equation $(CMAF)_{X,\omega, \mu, F}$ pointwise in $]0,T[ \times X$.

Define $F_\beta(t,x,r):= F(t,x,r-\Phi_{\beta}(x))$ and  $\mu_{\beta}(t,x):= e^{-\frac{\partial \Phi_{\beta}(t,x)}{\partial t}} \mu(t,x)  $ where $(t,x) \in ]0,T[ \times U_\beta$. 

\begin{defi}
A function $\phi:X_T\to \R$ is a viscosity sub/super-solution of $(CMAF)_{\omega, \mu, F}$ iff, 
for each $\beta\in B$,  $\psi=\phi+\Phi_\beta$ is  a viscosity sub/super-solution of the following parabolic Monge-Amp\`ere equation:
 $$
(PMA)_{\mu_\beta, F_\beta} \quad (dd^c\psi)^n = e^{\frac{\partial \psi}{\partial t} + F_{\beta} (t,x,\psi)} \mu_\beta  \quad \mathrm{on} \quad ]0,T[ \times U_\beta.
$$ 

A bounded function $\phi:X_T \rightarrow \R$ is a viscosity solution of $(CMAF)_{\omega, \mu, F}$ iff
it is both a sub- and a supersolution  of $(CMAF)_{\omega, \mu, F}$. Such a function is continuous.

A bounded function $\phi:X_T \rightarrow \R$ is a discontinuous viscosity solution of $(CMAF)_{\omega, \mu, F}$ iff
its upper semi-continuous regularisation $\phi^*$ is  a subsolution  of $(CMAF)_{\omega, \mu, F}$ and its lower semicontinuous regularisation $\phi_*$ is a supersolution of $(CMAF)_{\omega, \mu, F}$ .
\end{defi}

If a viscosity solution (resp. subsolution, resp. supersolution) is of class $C^{1,2}$, it is a classical solution (resp. subsolution, resp. supersolution). 
We refer the reader to \cite{EGZ14} for a study of viscosity sub/super-solutions to local complex Monge-Amp\`ere flows. 

This definition is a special case of the general theory of \cite{CIL92} for viscosity solutions of general
degenerate elliptic/parabolic equations. The reader is referred to this survey article for the first principles 
of the theory. The basic fact we certainly need to  recall is that subsolutions are u.s.c whereas supersolutions are l.s.c.

Recall that if $\omega$ is a closed smooth $(1,1)$-form in $X$, then the complex Monge-Amp\`ere  measure $(\omega + dd^c \p)^n$ 
is well-defined {\it in the pluripotential sense} for all bounded $\omega$-psh functions $\p$ in $X$, as follows from the work of  Bedford and Taylor (see \cite{BT76}, \cite{GZ05}) and  viscosity 
(sub)solutions of complex Monge-Amp\`ere equations can be interpreted in pluripotential theory as explained in \cite[Theorem 1.9]{EGZ11}. 

On the other hand, it is not clear to us how to interpret viscosity solutions of Complex Monge-Amp\`ere flows in terms of pluripotential theory. We will note however the following useful lemma
which follows easily from \cite[Theorem 1.9, Lemma 4.7]{EGZ11}. 

\begin{lem}\label{lem:pluripotvisc}
 Let $u\in C^0(X_T, \R)$ such that:
 \begin{itemize}
  \item $u$ admits a continuous partial derivative $\partial_t u$ with respect to $t$,
  \item For every $ t\in ]0,T[$, the restriction $u_t$ of $u$ to $X_t$  satisfies $$(\omega_t+dd^c u_t)^n \ge e^{\partial_t u+F(t,x,u)}\mu(t,x)$$ in the pluripotential sense on $X_t$.
 \end{itemize}

Then $u$ is a subsolution of $(CMAF)_{\omega, \mu, F}$.

 Let $v\in C^0(X_T, \R)$ such that:
 \begin{itemize}
 \item The restriction $v_t$ of $v$ to $X_t$ is $\omega_t$-psh,
  \item $v$ admits a continuous partial derivative $\partial_t v$ with respect to $t$,
  \item There exists a continuous function $w$ such that, for every $ t\in ]0,T[$, the restriction $v_t$ to $X_t$  satisfies 
  $$
  (\omega_t+dd^c v_t)^n \leq e^w \mu(t,x)
  $$ 
  in the pluripotential sense on $X_t$
  and ${\partial_t v_t+F(t,x,v_t)} \ge w$.
 \end{itemize}

Then $v$ is a supersolution of $(CMAF)_{\omega, \mu, F}$. 
\end{lem}

We also need to record a
basic property  from \cite{EGZ14}. 

\begin{prop}
Let $\phi$ be a  viscosity subsolution of $(CMAF)_{\omega, \mu, F}$. For each $t\in ]0,T[$, we have 
$\phi_t \in \mathrm{PSH}(X, \omega_t)$. 
\end{prop}

Let us remark  that,
for these applications of the results of \cite{EGZ14} on weak solutions to local complex Monge-Amp\`ere flows, it is actually enough
to assume the technical condition that $\Phi_{\beta}$ is continuous and locally  Lipschitz  in  the  time variable  (hence
$\frac{\partial}{\partial t}  \Phi_{\beta}$  exists  a.e. ) and  $\mu$ is only measurable but:

\begin{equation} 
\mu_{\beta}:= e^{-\frac{\partial \Phi_{\beta}}{\partial t}} \mu |_{U_\beta} {\mathrm{ \ is \ continuous.}}   \label{vreg}
\end{equation}

We note that this condition allows the function $t \mapsto \{\omega_t\}$ be Lipschitz  non-differentiable. However,
we are not able to prove global results unless the following stronger regularity condition holds:

\begin{equation} 
\Phi_{\beta}\in C^{1,2}, \ \mu {\mathrm{ \ is \ continuous.}}  \label{wreg},
\end{equation}

which is indeed what Definition \ref{cmafdef} requires.

\subsection{The K\"ahler Ricci flow with canonical singularities}

\subsubsection*{Normalized K\"ahler-Ricci flow}

Let us now interpret in the present framework the K\"ahler-Ricci flow on varieties with canonical singularities that was 
defined in \cite{ST2}. 

Let $Y$ be an irreducible normal compact K\"ahler space with only canonical singularities and $n=\dim_{\C}(Y)$. 
Let $\pi: X\to Y$
be a log-resolution, i.e.: $X$ is a compact K\"ahler manifold, $\pi$ is a bimeromorphic projective morphism
and $\mathrm{Exc}(\pi)$ is a  divisor with simple normal crossings. Denote by $\{E\}_{E\in \mathcal{E}}$ the family
of the irreducible components of  $\mathrm{Exc}(\pi)$.  With this notation, one has furthermore
$$K_X\equiv \pi^* K_Y + \sum_E a_E  E  $$
where $a_E \in \Q_{\ge 0}$,  $K_Y$ denote the first Chern class in Bott-Chern cohomology of the $\Q$-line bundle  $O_Y(K_Y)$ on $Y$  whose restriction to the smooth locus is 
the 
line bundle whose sections are holomorphic top dimensional  forms (or according to the standard terminology canonical forms), $K_X$ the canonical class of $X$ 
and $E$ also denotes with a slight abuse of language the cohomology class of $E$.
This means that for every non-vanishing locally defined multivalued canonical form $\eta$ defined over $Y$, 
the holomorphic multivalued canonical form $\pi^* \eta$ on $X$ has a zero of order $a_E$ along $E$. 

Denote by $\mathcal{K}(Y) \subset H^{1}(Y, \mathcal{PH}_Y)$ the
open convex cone of K\"ahler classes and
let $\omega_0$ be a semi-K\"ahler form  on $Y$ with $C^2$ potential (see \cite{EGZ09} for the definitions of K\"ahler metrics 
and variants on normal complex spaces) such that $\{\omega_0\}+ \epsilon K_Y \in \mathcal{K}(Y)$ for $1\gg \epsilon >0$.  Assume $h$ is a 
smooth hermitian metric on the holomorphic $\Q$-line bundle underlying $O_Y(K_Y)$. 
Then 
$$
\chi:=-dd^c \log h
$$
is a smooth representative of $K_Y\in H^{1}(Y, \mathcal{PH}_Y)$. 
\smallskip

We are going to study the existence and the long term behavior of the normalized K\"ahler-Ricci flow (NKRF for short) on
$Y$,
$$
\frac{\partial \omega_t}{\partial t}=-\rm{Ric}(\omega_t)-\omega_t,
$$
starting from the  initial data $\omega_0$. At the cohomological level, this yields
a first order ODE showing that the cohomology class of $\omega_t$ evolves as
$$
\{ \omega_t \} =e^{-t} \{ \omega_0 \} +(1-e^{-t} ) K_Y.
$$

We thus let $T_{max}\in ]0,+\infty]$ be defined by
$$
T_{max} :=\sup\{ t>0,  \ e^{-t}\{\omega_0\}+(1-e^{-t}) K_{Y} \in \mathcal{K}(Y) \}
$$
and denote by 
the following $C^1$ in $t\in [0,T[$ relative semi-K\"ahler form on $Y_T$,
$$
\chi_t= e^{-t} \chi_0+ (1-e^{-t}) \chi, 
$$
where $\chi_0$ is a smooth K\"ahler representative of the K\"ahler class ${\omega_0}$ and $\chi$ is a smooth representant of the canonical class $K_Y$.

Then $\omega=(\omega_t)_{t\in [0,T[}$ the solution of the normalized K\"ahler-Ricci flow can be written as $\omega_t = \chi_t + dd^c \phi_t$, where $\phi : Y_T \longrightarrow \R$ is continuous in $Y_T$.

We now define $$\omega_{NKRF}:=\pi^*\omega \in H^0(X, \mathcal{Z}^{1,1}_{X_T/[0,T[})$$ and  $$\mu_{NKRF}= c_n
\frac{\pi^*\eta \wedge \overline{ \pi^*\eta}}{ \pi^* \| \eta \|^2_h} \in C^0(X,\Omega^{n,n}_X)$$
which we view as a  continuous element of $C^0(X_T, \Omega^{n,n}_{X_T/[0,T[} )$ and $c_n$ is the unique
 complex number of modulus $1$ such that the expression is positive. 
As the notation suggests, $\mu_{NKRF}$ is independent of the auxiliary multivalued holomorphic form $\eta$ but depends on $h$ an auxiliary smooth metric on $Y$.
When it will be necessary to display this dependance we shall write $\mu_{NKRF}(h)$. 
Since the local potentials of $\chi$ are of class $C^{\infty}$ the pair $(\omega_{NKRF}, \mu_{NKRF})$ satisfies 
the requirements of Definition \ref{cmafdef}. 

In local coordinates $\mu_{NKRF}$ has a continuous density of the form 
$$
v_{NKRF}=\prod_E |f_E|^{2a_E} v
$$ where $v>0$ is smooth and 
$f_E$ is an equation of $E$ in these local coordinates. 

\begin{lem} \label{reso} 
Every viscosity solution $\phi_{\pi}$  of the Monge-Amp\`ere flow
$$(CMAF)_{X,\omega_{NKRF}, \mu_{NKRF}, r} \ (\omega_{NKRF}+dd^c\phi)^n=e^{\phi+\frac{\partial\phi}{\partial t}} \mu_{NKRF} $$  
with Cauchy datum $ \phi_0$ descends to $Y_T$, i.e.: $\phi_{\pi}=\pi^*{\phi}$
and the element $\omega+dd^c\phi \in H^0(Y, \mathcal{C}^0_{Y_T}/\mathcal{PH}_{Y_T/[0,T[})$  obtained this way is independent
of $\pi$ and of $h$. 
\end{lem}

\begin{proof} The fact that $\omega+dd^c\phi $ does not depend on the auxiliary 
hermitian metric $h$ is obvious. 
The rest follows from the quasi-plurisubharmonicity of viscosity (sub)solutions established in \cite{EGZ14},
together with the argument in \cite{EGZ09} for the static case which implies that $\phi_{\pi}$ is constant along the fibers of $\pi$. 
This also works for subsolutions. 
\end{proof}

 \begin{defi}\label{def:nkrf}
We say that $\omega+dd^c\phi \in H^0(Y, \mathcal{C}^0_{Y_T}/\mathcal{PH}_{Y_T/[0,T[})$ as in Lemma \ref{reso} is a solution of the normalized K\"ahler-Ricci flow on $Y$ starting at $\omega_0$. 
 \end{defi}

Lemma \ref{reso} implies that the notion does not depend on the choice of the log resolution $\pi:Y\to X$. 

A basic observation is that the normalized K\"ahler-Ricci flow can also be formulated as 
$(CMAF)_{\omega', \mu'}$ where
$$
\omega'=\pi^*\omega+dd^c\Psi,  \; \; \mu'=e^{-\Psi-\frac{\partial \Psi}{\partial t}}\mu_{NKRF},
$$ 
 $\Psi \in C^{\infty}(X_T,\R)$ being arbitrary.

\subsubsection*{Change of time variable}
The important case of $(CMAF)_{\omega, \mu, F}$ when dealing with the K\"ahler-Ricci flow is thus when $F (t,x,r) = \alpha r$. Then 
the sign of $\al$ is crucial for the long term behavior of $\f_t$. We however observe that it
plays no role for finite time:

\begin{lem}\label{chgvar}
 The solutions of the flows:
\begin{eqnarray*}
 (\omega_1(t)+dd^c\phi)^n&=&e^{\alpha \phi +\frac{\partial \phi}{\partial t}} \mu_1(t), \ t\in[t_0, t_1[ \\
(\omega_0(s)+dd^c\psi)^n&=&e^{\frac{\partial \psi}{\partial s}} \mu_0(s), \ s\in[s_0,s_1[
\end{eqnarray*}
coincide if we do the following change of variables (when $\alpha > 0)$: 
$$
s- s_0 =e^{\alpha(t-t_0)}-1, \ \   \psi(s)=(1 + s-s_0)\phi \left(t_0+
\frac{\log(1 + s-s_0)}{\alpha } \right)
$$ 
where
\begin{eqnarray*}
\omega_0(s)&=& (1 + s-s_0) \, \omega_1 \left(t_0+\frac{\log(1 + s-s_0)}{\alpha } \right),\\
 \mu_0(s)&=&(1 + s-s_0)^n \, \mu_1 \left(t_0+\frac{\log(1 + s-s_0)}{\alpha } \right).
\end{eqnarray*}    
\end{lem}

The proof is a straightforward computation. In the sequel we will therefore often reduce to the case $\al=0$.

\subsubsection*{K\"ahler-Ricci flow}

The above formulation of the normalized K\"ahler-Ricci flow is adapted to the asymptotic 
behaviour of solutions when $t\to +\infty$.  Applying Lemma \ref{chgvar} below to this flow, one gets
another equivalent flow which is nothing but  the (classical,  unnormalized) K\"ahler-Ricci flow. 
This equivalent formulation
is given in the definition:

\begin{defi}\label{krf}
A  flow on $X_T$ of the form $(\omega+ dd^c\phi)^n=e^{\frac{\partial \phi}{\partial t}} \mu$ 
is a K\"ahler-Ricci flow on $Y$ iff 
\begin{itemize}
\item $\omega_t \in \pi^* \mathcal{K}(Y)$ for $t>0$;
 \item $\mu=\prod_E |s_E |_{h_{E}}^{2a_E} W$ where $W$ is a volume form with continuous positive density on $X$, $s_E \in H^0(X, O_X(E))$ denotes the tautological section and $h_E$ a smooth metric on $O_X(E)$;
\item $\frac{\partial \omega}{\partial t}= dd^c\log \mu -a_E [E]$ in the sense of currents. 
\end{itemize}

\end{defi}

The corresponding cohomological flow takes the form $\frac{\partial \{ \omega_t \}}{\partial t}= \pi^* K_Y$.

 \subsubsection*{Klt pairs}
Let us also mention without going into details that one may also replace $Y$ with a  pair $(Y,\Delta)$ having  klt singularities. In that case, 
$$K_X\equiv \pi^* (K_Y+\Delta) + \sum_E a_E  E  $$ with $a_E>-1$. 
In that case $\mu_{NKRF}$ has poles and the preceding discussion does not apply. 
However, using a construction of \cite{EGZ09},  allowing high ramification along the $E$'s, we construct   a compact complex orbifold $\mathcal{X}$ whose moduli space is $c:\mathcal{X} \to X$
and we may do the preceding construction replacing $X$ by $\mathcal{X}$. Indeed $c^*\mu_{NKRF}$ is continuous in orbifold coordinates.

\subsection{The Perron discontinuous viscosity solution}

A very attractive feature of discontinuous viscosity solutions is that their existence is easily established.

\begin{defi}

A Cauchy datum for  $(CMAF)_{\omega, \mu, F}$ is  a continuous function $\phi_0: X\to \R$ such that $\phi_0 \in \mathrm{PSH}(X, \omega_0)$.

We say $\phi\in USC(X_T, \R\cup{-\infty})$ (resp. $LSC(X_T, \R\cup{+\infty})$) is a subsolution (resp. supersolution) to the Cauchy problem: 
$$ (\omega +dd^c\phi)^n= e^{\frac{\partial \phi}{\partial t} + F (t,x,\phi)} \mu, \ \ \phi|_{X \times \{0 \}}=\phi_0 \leqno (CMAF)_{X,\omega, \mu, F}(\phi_0)$$
if $\phi$ is a subsolution (resp. supersolution)  to  $(CMAF)_{\omega, \mu, F}$ such that $\phi|_{X\times \{0 \}}\le \phi_0$  (resp. $\ge$).

The Cauchy problem  $(CMAF)_{X,\omega, \mu, F}(\phi_0)$ is said to be admissible if
it has a bounded subsolution and there exists  a continuous function $\psi$  such that $\phi_0 \le \psi |_{X \times \{0 \}}$ and 
every subsolution is $\le \psi$.

\end{defi}

For instance, if $(CMAF)_{X,\omega, \mu, F}(\phi_0)$ admits a classical {\em strict} supersolution $\psi$, 
this Cauchy problem is admissible.

\begin{prop} \label{perron}
 If the Cauchy problem  $(CMAF)_{X,\omega, \mu, F}(\phi_0)$ is admissible, denoting by $\mathcal{S}$ the set of all its subsolutions,  the usc regularization $s^*$
 of $s:=\sup_{u\in \mathcal{S}} u$ is a discontinuous viscosity solution of $(CMAF)_{X,\omega, \mu, F}$. 
\end{prop}

\begin{proof}
Omitted. See \cite{CIL92}, \cite{IS12}. 
\end{proof}

This construction raises two issues: whether $s^*$ is continuous, hence a true viscosity solution and whether it is a solution to the Cauchy problem 
in the na\"ive sense namely whether $s^*|_{X \times \{0 \} }=\phi_0$. The first issue is generally treated using a Comparison Principle and the second issue 
is taken care of by barrier constructions. 

The Parabolic Comparison Principle (PCP) states that if $\phi$ (resp. $\psi$)
is a subsolution (resp. a supersolution)
to $(CMAF)_{X,\omega, \mu, F}(\phi_0)$ then $\phi \le \psi$.  It implies that $s^*$ as in Proposition  \ref{perron} is the unique  viscosity solution  to $(CMAF)_{X,\omega, \mu, F}(\phi_0)$
and that it is continuous.  If the (PCP) holds for sub/supersolutions with extra regularity (e.g.: classical, Lipschitz, ...)
it implies that there is at most one viscosity solution with this extra regularity. 

Comparison Principles
are rather elaborate forms of the maximum
principle. We believe that (PCP) should hold under condition (1.1) provided there exists a semipositive smooth closed $(1,1)$-form $\theta$ of positive volume such that $\omega_t \ge \theta$ for all $t\in [0,T[$. 
Unfortunately, proving (PCP) is rather technical and we will describe what we have been able to prove in the next section. 
It is  very encouraging that  optimal results in that direction are available in the local case \cite{EGZ14}.

\section{Parabolic Comparison Principles}

  Let $X$ be a compact complex manifold of dimension $n$ and $\omega_t$ a continuous family of closed real $(1,1)$-forms on $X$.
We consider the complex Monge-Amp\`ere flow on $X_T= [0,T[ \times  X $ associated to $(\omega,F,\mu)$,
\begin{equation} \label{eq:PCMA}
 e^{\dot{\f}_t  + F (t,x,\f) } \mu (t,x) - (\omega_t +dd^c \f_t)^n = 0, 
 \end{equation}
 according to definition \ref{cmafdef}. 
\smallskip

\subsection{Statement of the Global Parabolic Comparison principles}
Let $\f$ (resp. $\p$) be a bounded subsolution (resp. supersolution) to the parabolic complex Monge-Amp\`ere equation (\ref{eq:PCMA}) in $X_T$ associated to $(\omega, \mu ,F)$. 
Our goal in this section is to establish several versions of the  global comparison principle,
starting with the following:

\begin{theo}\label{SCP}
 Assume that $\mu (t,x) > 0$ is positive in $X_T$ and that locally in $]0,T[ \times X$ we have the inequality  $  \partial_t \p  \geq   - C$ in the sense of viscosity, for some constant $C >0$.
  Then for all $(t,x) \in [0,T[ \times X$,
$$
\f (t,x) - \p (t,x) \leq \max_{x \in X} (\f (0,x)  - \p (0,x))_+.
$$
In particular if $\f (0,x) \leq \p (0,x)$ in $ X$ then $\f (t,x) \leq \p (t,x)$ in $[0,T[ \times X$.
 \end{theo}

Observe that as in the local case (see \cite{EGZ14}, Remark 2.4), in order to apply the parabolic Jensen-Ishii's maximum principle, we need to assume that the supersolution satisfies a
local lower bound on its time derivative even when $\mu > 0$, since our parabolic equation has a structural dissymetry.

The versions of the comparison principle that we will use in the sequel require that one weakens the hypothesis that $\mu$ be positive and lift the 
condition on $\psi$. It is not clear what is the optimal provable result in this direction. We will state and prove four variants 
which all 
require that we strengthen our working hypotheses as follows:

\begin{displaymath}
\label{Condition} (\ref{Condition})
\left\{
\begin{array}{ll}
a) &X \ {\mathrm{is \ K\ddot{a}hler}}, \\
b) & {\mathrm{There \ exists \ a \ semipositive \ closed \ }} (1,1) {\mathrm{\ form \ }} \theta \ \mathrm{on} \ X \ {\mathrm{such \ that \ }} \\
&\omega_t\ge \theta \  {\mathrm{and}}  \ \{\theta\}^{n}>0, \\
c) &(t,x,r)\mapsto F(t,x,r) \  {\mathrm{is \  uniformly \ Lipschitz \  in  \ the \ }} r {\mathrm{\ variable}},\\
d)&(t,x)\mapsto F(t,x,0) \ {\mathrm{is  \ uniformly \ bounded \ above.}}
\end{array}\right.
\end{displaymath}

Since we also require $F$ is non decreasing in the $r$ variable, the conditions for $F$ are satisfied when $F(t,x,r)=\alpha r$ with $\alpha\ge 0$.

\begin{coro} \label{Cor1} Assume (\ref{Condition}) holds. 
 Assume that $\mu (t,x) \geq 0$ in $X_T$ and locally in $]0,T[ \times X$, there exists a constant $C > 0$ such that  $\vert \partial_t \f\vert  \leq  C$ and  $\partial_t \p  \geq  - C$.
Then 
$$
\f (t,x) - \p (t,x) \leq \max_{x \in X} (\f (0,x)  - \p (0,x))_+,
$$
for all $(t,x) \in [0,T[ \times X$.
\end{coro}

 \begin{coro} \label{Cor2}  Assume (\ref{Condition}) holds.    
 Assume that $\mu (t,x)=\mu(x)  \geq 0 $ and $t \longmapsto \omega_t = \omega (t,\cdot)$ is constant. 
Then for all $(t,x) \in [0,T[ \times X$,
$$
\f (t,x) - \p (t,x) \leq \max_{x \in X} (\f (0,x)  - \p (0,x))_+.
$$
\end{coro}

The following generalization holds:

 \begin{coro} \label{Cor3}  Assume (\ref{Condition}) hold.    
 Assume that $\mu (t,x)=\mu(x)  \geq 0 $ and $t \longmapsto \omega_t = \omega (t,\cdot)$ is monotone in $t$. 
Then for all $(t,x) \in [0,T[ \times X$,
$$
\f (t,x) - \p (t,x) \leq \max_{x \in X} (\f (0,x)  - \p (0,x))_+.
$$
\end{coro}
 
In order to lift this monotonicity condition, we introduce a slightly technical condition in addition to (\ref{Condition}). 

\begin{defi} \label{defi:oregular} Say $t\mapsto \omega_t$ is regular, if the following holds: 

For every positive real constant $ \varepsilon >0$ there exists $E(\e)> 0 $ such  that
$$ 
\forall t\in [0,T-2\e],  \ \forall t'\in ]t-\e,t+\e[,  \ (1+E(\varepsilon)) \omega_t \ge \omega_{t'} \ge (1-E(\varepsilon) ) \omega_t 
$$

and $E(\e) \to 0$ as $\e \to 0$. 
 
\end{defi}

 \begin{coro} \label{Cor4}  Assume that (\ref{Condition}) holds,  $t\mapsto \omega_t$ is regular in the sense of Definition \ref{defi:oregular}    
 and $\mu (t,x)=\mu(x)  \geq 0 $. 
Then for all $(t,x) \in [0,T[ \times X$,
$$
\f (t,x) - \p (t,x) \leq \max_{x \in X} (\f (0,x)  - \p (0,x))_+.
$$
\end{coro}

Regularity in the sense of Definition \ref{defi:oregular}   holds true if, in addition to (\ref{Condition}), we require $\theta$ to be K\"ahler. 
It should be remarked that when $\omega_t$ is a smooth family of K\"ahler forms and $\mu$ is a smooth positive volume form, the optimal comparison principle 
holds true and follows from the existence of a classical solution to the Cauchy problem. 
However, one needs Corollary \ref{Cor4} to obtain it by the present methods, 
the other versions being too weak.  
On the other hand, the following generalization of this remark covers many cases of interest:

\begin{lem} 
Let $\pi: X\to Y$ be a bimeromorphic morphism onto a normal K\"ahler variety. Assume $\omega_t^Y$ is a continuous family of {\bf{smooth}}
K\"ahler forms on $Y$. Then $\omega_t=\pi^*\omega_t^Y$ satisfies (\ref{Condition}) and is regular in the sense of Definition \ref{defi:oregular}. 
\end{lem}

\subsection{Proofs}
We start by proving the Theorem and give the proof of the corollaries afterwards.
 
 \begin{proof}

We first establish a slightly more general estimate (\ref{eq:FIneq0}) assuming $\mu > 0$
is positive.

Namely let   $\mu (t,x) >  0$ and $\nu (t,x) \ge 0$ be  two positive continuous  volume forms  on $X_T$  and $F,G : \R^+  \times X \times \R \longrightarrow \R$ two continuous functions.
Let $\f$  be a bounded subsolution to the  parabolic complex Monge-Amp\`ere equation (\ref{eq:PCMA}) associated to $(\omega,F,\mu)$  in $X_T$ and $\p$ be a bounded  supersolution to the parabolic 
complex Monge-Amp\`ere equation (\ref{eq:PCMA}) associated to $(\omega,G,\nu)$  in $X_T$. We assume furthermore that $\partial_t \psi \ge -C$ locally on $X_T$. 
  
We are going to show that for any fixed $\delta > 0$ small enough, either there exists a point  $(\hat t, \hat x) \in ]0,T[ \times X$  where
the function defined by 
$$
\tilde \f (t,x) - \p (t,x) := \f (t,x) - \frac{\delta}{T -t} - \p (t,x) 
$$
 achieves its maximum on $X_T$ and the following inequality is satisfied
\begin{equation} \label{eq:FIneq0}
 e^{\frac{\delta}{(T -\hat t)^2} +  F (\hat t,\hat x, \tilde \f (\hat t,\hat x))}  \mu (\hat t,\hat x) \leq e^{ G (\hat t,\hat x, \p (\hat t,\hat x))} \nu (\hat t,\hat x),
\end{equation}
or this maximum is achieved at some point $(0, \hat x)$ on the parabolic boundary. 
This is a global version of (\cite{EGZ14}, Lemma 3.1).

 Choose a large constant $C>0$ such that $\f$ and $\p$ are both  $\le C/ 4$ in $L^{\infty}$-norm and fix $\delta>0$ arbitrarily small.

Since $ \tilde \f - \p$ is upper semicontinuous in $[0,T[ \times X$ and tends to $- \infty$ when $t \to T^-$, the maximum of  $\f - \p$ is achieved at some point $(t_0,x_0) \in [0,T[ \times \Omega$ i.e.
$$
M := \sup_{(t,x) \in [0,T[ \times X} (\tilde \f (t,x) - \p  (t,x)) = \tilde \f (t_0,x_0) - \p  (t_0,x_0).
$$
and there exists $T' < T$ such that it cannot be achieved in $ [T',T[ \times X$ i.e.$t_0 \in [0,T'[ \times X$.

If $t_0 = 0$ then we obtain  for any $(t,x) \in X_T$,
\begin{equation} \label{eq:Case1}
 \tilde \f (t,x) - \p  (t,x) \leq M = \tilde \f (0,x_0) - \p  (0,x_0) = \max_{X} (\tilde \f (0,x) - \p  (0,x)).
\end{equation}
\smallskip

We now focus on the most delicate case when $t_0 \in ]0,T'[$ and assume that the maximum  $M$ of $\tilde \f (t,x) - \p  (t,x)$ is not achieved in $\{0\} \times X$, nor in $ [T',T[ \times X$ i.e.
\begin{equation} \label{eq:Strict}
M  >  \max_{X'_T} \{\tilde \f (t,x) - \p (t,x) \},  \, \text{where} \, X'_T := \{0\} \times X \cup [T',T[ \times X
\end{equation}

The idea is to localize near the point $x_0$ and apply the parabolic Jensen-Ishii's maximum principle from \cite{EGZ14}). 
Choose complex coordinates 
$z = (z^1, \ldots, z^n)$ near $x_0$
defining a biholomorphism identifying an open neighborhood of $x_0$ to the complex ball $B_4 := B(0,4) \subset \C^n$ of radius $4$, sending $x_0$ to the origin in $\C^n$. 

Observe that $\tilde \f$ is upper semi-continuous and satisfies,  in  $X_T = ]0,T[ \times X$, 
the viscosity differential inequality
 $$
  e^{\partial_t \tilde \f +  \frac{\delta}{(T -t)^2} +  F (t,x,\tilde \f +  \frac{\delta}{T -t})}  \mu (t,x)  \leq  (\omega +dd^c \tilde \f)^n.
 $$
We let $h (t,x)$  be a continuous local potential for $\omega$ such that $\partial_t h$ is continuous in $[0,T[ \times B_4$ i.e.  $dd^c  h = \omega $ in $[0,T[ \times B_4$. 
We may without loss of generality assume that $C$ is choosen so large that $\| h \|_{\infty} < C/ 4$.

Consider the  upper semi-continuous function 
$$
 \tilde u (t, \zeta) :=  \tilde \f (t,z^{- 1} (\zeta)) + h (t, z^{- 1} (\zeta)),
$$

Then $\tilde u$ satisfies the viscosity differential inequality 
\begin{equation} \label{eq:subsol}
e^{\partial_t \tilde u  + \frac{\delta}{(T -t)^2} + \tilde F (t,\zeta,  \tilde u)} \tilde \mu (t,\zeta)  \leq 
(dd^c \tilde u)^n, \, \, 
\text{ in } ]0,T[ \times B_4,
\end{equation}
where 
$\tilde \mu :=  z_{*} (\mu)\geq 0$ is a  continuous volume form on $B_4$ and 
$$
\tilde F (t,\zeta,r) = F \left(t,x, r - h (t, x)  + \frac{\delta}{T -t}\right) - \partial_t h (t, x), 
$$
where $x := z^{- 1} (\zeta).$

In the same way,  the lower semi-continuous function 
$$
v (t, \zeta) := \p (t,z^{- 1} (\zeta)) + h (t,z^{- 1} (\zeta))
$$
 satisfies the viscosity differential inequality
\begin{equation} \label{eq:supersol}
e^{\partial_t v + \tilde G (t,z,v)} \tilde \nu  (\zeta)  \geq (dd^c v)^n, \, \, 
\text{ in } B_4,
\end{equation} 
where $\tilde \nu :=  z_{*} (\nu) \geq 0$ is a positive and continuous volume form on $B_4$ and 
 $\tilde G \left(t,\zeta,r) = G (t,x,r - h (t,x)\right) - \partial_t h (t, x)$, with $x :=  z^{- 1} (\zeta)$. 

Observe that the functions $\tilde F$ and $\tilde G$ are continuous in $[0,T[ \times B_4$ since $\partial_t h$ is continuous.

\smallskip

Then we have 
\begin{equation} \label{eq:LocMax}
M  =  \tilde u (t_0,0) - v (t_0,0)  =  \max_{[0,T'] \times  \bar B_3} (\tilde u (t,\zeta) -  v(t,\zeta)).
\end{equation}

We are going  to estimate the number $M$ by applying the parabolic version of Jensen-Ishii's maximum principle. 

As in the local case we use a penalization method (\cite{CIL92}, \cite{EGZ14}) but we need the localizing trick of \cite{EGZ11} which consists in introducing a new localizing penalization function. For $\e>0$, we consider the function defined in 
$[0,T[ \times B_4 \times B_4$ by
 $$
 (t,x,y) \longmapsto \tilde  u (t,x) - v  (t,y)  - \sigma (x,y) - (1 \slash 2 \e) \vert x - y\vert^2, 
 $$
 where   $\sigma $ is the localizing penalization function constructed in \cite{EGZ11}. This a non negative  smooth function $\sigma (x,y) \geq 0$ in $X^2$ which vanishes to high order only on the diagonal near the origin $(0,0)$ and is large enough on the boundary of the ball $B_3 \times B_3$ so that $\sigma \geq 3 C$ on $\bar  B_4^2 \setminus B_2^2$, to force the maximum to be attained at an interior point.

The role of function $\sigma$ is to force the maximum to be asymptotically attained  along the diagonal (as in the degenerate elliptic case, see \cite{EGZ11}). The fact that
 the second derivative of $\sigma$ is a quadratic form  on $\R^{2n} \times \R^{2 n}$ which vanishes on the diagonal is going to be crucial in the sequel.

\smallskip

Observe that since $\sigma (0,0) = 0$, we also have
\begin{equation} \label{eq:theta}
M =  \max_{[0,T'] \times  \bar B_3} (\tilde u (t,\zeta) -  v(t,\zeta) - \sigma (\zeta,\zeta).
\end{equation}

Since we are maximizing an upper semi-continuous function on the compact set $[0,T[ \times \bar B_3^2$,
there exists $(t_\e,x_{{\e}},y_{\e}) \in [0,T[ \times \bar B_3 \times \bar B_3$ such that
\begin{eqnarray*}
M_{{\e}}& := &\sup_{(t,x,y) \in [0,T'] \times \bar B_3^2} \left\{\tilde u (t,x)- v (t,y) - \sigma(x,y) -\frac{1}{2 \e} \vert x - y\vert^2 \right\}\\
&=& \tilde u (t_\e,x_{{\e}})- v  (t_\e,y_{{\e}}) -\sigma(x_{{\e}},y_{{\e}})-\frac{1}{2\e} \vert x_{\e} - y_{\e}\vert^2.
\end{eqnarray*}
Observe that $\f, \p, h$ are bounded by $C \slash 4$ in the $L^{\infty}$-norm in $[0,T[ \times \bar B_4$, while 
$ \sigma \geq 3 C $ on $\bar B_4^2 \setminus B_2^2$. 
Therefore for any $\e$, we have
   
 \begin{equation} \label{eq:Min}
   M_{\e} \ge M =  \max_{\bar B_3} ( \tilde  u (0,x) - v (0,x)) \geq - 3  C \slash 4 - \delta \slash T.
 \end{equation}
On the oter hand,  for $(t,x,y) \in  [0,T[ \times B_3^2 \setminus B_2^2$, we have
 \begin{equation} \label{eq:Max}
 u (t,x)- v (t,y) - \sigma(x,y) -\frac{1}{2 \e} \vert x - y\vert^2  \leq  +  C - 3 C  = - 2 C.
 \end{equation}
 Therefore if we assume  $0 < \delta < C T \slash 4$, then 
 for any $\e > 0$ small enough, we have  $(t_\e,x_{{\e}}, y_{{\e}}) \in [0,T'] \times B_2^2$. 

 The  following result is classical (see \cite[Proposition 3.7]{CIL92}):

 \begin{lem} \label{asymp}
We have $ \vert x_{\e} - y_{\e}\vert^2  = o ({\e})$. Every limit
point $(\hat t,\hat x, \hat y)$ of $(t_\e,x_{\e}, y_{\e})$ satisfies $\hat x= \hat y$, $(\hat x , \hat x) \in \Delta \cap \bar B_2^2$, $ \hat t \in [0,T]$ and 
$$
\lim_{{\e} \to 0} M_\e = \lim_{{\e} \to 0}  (\tilde  u (t_\e,x_{\e})- v (t_\e,y_{\e}) - \sigma (x_{\e},y_{\e}) 
=  \tilde  u (\hat t,\hat x)- v (\hat t,\hat x) - \sigma (\hat x,\hat x).
$$
Moreover  $(\hat t,\hat x) \in ]0,T[ \times B_2$.
\end{lem}
\begin{proof}

 Observe that the first part of lemma is a consequence of \cite[Proposition 3.7]{CIL92}. To prove the second part we use following easy observation. From the first part of the lemma, using  (\ref{eq:theta}) and (\ref{eq:Min}), we deduce that 
$$
M \leq \lim_{{\e} \to 0} M_\e =  \tilde  u (\hat t,\hat x)- v (\hat t,\hat x) - \sigma (\hat x,\hat x) \leq M - \sigma (\hat x,\hat x),
$$
hence $ \sigma (\hat x,\hat x) = 0$. 
Since by construction $\Delta \cap \sigma^{-1} (0) \subset B_2^2$, it follows that $\hat x \in B_2$.
\end{proof}

It follows from  (\ref{eq:Strict}) that $(\hat t,\hat x) \in ]0,T'[ \times B_2$, hence  it is an interior point of $[0,T'] \times \bar B_3^2$. Thus there exist a sequence $(t_{\e_j},x_{\e_j},y_{\e_j}) \in ]0,T'[ \times B_2$ which converges to $(\hat t, \hat x)$ such that the conditions of the Lemma are satisfied.
 
We now apply the parabolic Jensen-Ishii's  maximum principle (see \cite{EGZ14}) to $u$ and $v$  with  $\phi (t,x,y)  =\frac{1}{2\e} \vert x - y \vert^2 + \sigma (x,y)$. 
For $j >> 1$,
we get the following:

\begin{lem}\label{main}
For any $\gamma > 0$, we can find $(\tau^+_j,p_j^+, Q_j^+), (\tau_j^-,p^-_j, Q^-_j)\in \R \times \C^n \times Sym_{\R}^2 (\C^n)$
such that
\begin{enumerate}
 \item $(\tau_j^+,p_j^+, Q_j^+)\in \overline{\mathcal P}^{2+} u (t_{\e_j},x_{\e_j})$, $(\tau_j^-,p_j^-, Q_j^-)\in 
\overline{\mathcal P}^{2-} v (y_{\e_j})$, where 
\begin{eqnarray*}
p_j^+ &=& D_x \sigma (x_{\e_j},y_{\e_j}) + \frac{(x_{\e_j} - y_{\e_j})}{2 {\e_j}},  \\ 
p_j^- &= & - D_y \sigma (x_{\e_j},y_{\e_j}) - \frac{(x_{\e_j} - y_{\e_j})}{2 {\e_j}}, \\
\tau^+_j &= &\tau_j^- + \frac{\delta}{(T-t_{\e_j})^2}\cdot
\end{eqnarray*}
\item The block diagonal matrix with entries $(Q^+_j⁺, Q^-_j)$ satisfies:
$$
-(\gamma^{-1}+ \| A \| ) I \le 
\left(
\begin{array}{cc}
Q^+_j⁺ & 0 \\
0 & -Q^-_j
\end{array}
\right)
\le A+\gamma A^2, 
$$
where $A=D^2\phi(x_{\e_j}, y_{\e_j})$, i.e.
$$A =\e_j^{-1}
\left(
\begin{array}{cc}
I &  -I\\
-I &  I
\end{array}
\right) + D^2\sigma (x_{\e_j}, y_{\e_j})$$
and $\| A \|$ is the spectral radius of $A$ (maximum of the absolute values for the eigenvalues of this symmetric
matrix). 
\end{enumerate}
\end{lem}

\begin{proof} The proof, just like in \cite[section 3]{EGZ14}, consists in applying \cite[Theorem 8.3]{CIL92} with $u_1=u$,  $u_2=-v$. 
Observe that since the situation is  localized near $(\hat t, \hat x)$ in $]0,T'[ \times B_2$, the local lower bound on $\partial_t \p$ and a 
local lower bound on $\partial_t h$ implies a local lower bound  $ \partial_t v \geq - C$ near the point $(\hat t, \hat x)$ with a larger constant $C > 1$.
This permits to fulfill condition (8.5) on $-v$ when applying \cite[Theorem 8.3]{CIL92}. 
Also, we use $\mu>0$ to see that condition (8.5) in loc. cit. is satisfied by $u$. 
Observe also that $\tau_j^- \geq - C$ for $j >$ large enough (see the Remark following Proposition 1.6 in \cite{EGZ14}). 
\end{proof}

By construction, the Taylor series of $\sigma$ at any point in
 $\Delta \cap \sigma^{-1} (0)$ vanishes up to order $2n+2$.
In particular, 
$$
 D^2\sigma (x_{\e_j}, y_{\e_j}) =O(\vert x_{\e_j}- y_{\e_j}\vert^{2n})
=o(\e_j^{n}).$$
This implies $\|A \|\simeq 1 \slash \e_j$. We choose $\gamma = \e_j$ and deduce
$$
-(2\e_j^{-1} ) I \le 
\left(
\begin{array}{cc}
Q_j^+ & 0 \\
0 & -Q^-_j
\end{array}
\right)
\le \frac{3}{\e_j} \left(
\begin{array}{cc}
I &  -I\\
-I &  I
\end{array}
\right) + o(\e_j^{n})
$$

Looking at the upper and lower diagonal terms we deduce that the eigenvalues
of $Q_j^+, Q_j^-$ are positive and  $O(\e_j^{-1})$. Evaluating the inequality on vectors of the form 
$(Z,Z)$ we deduce that the eigenvalues
of $Q_{j}^+ - Q_j^{-}$ are $ \le  o(\e_j^{n})$.

\smallskip

For a fixed $Q\in Sym_{\R}^2 (\C^n)$, denote by $H = Q^{1,1}$ its $(1,1)$-part. It is a hermitian matrix. 
 Obviously the eigenvalues of $H_j^+ := (Q_j^+)^{1,1}, H_j^- := (Q^-)^{1,1}$ are $O(\e_j^{-1})$ but those
of $H_j^+ - H_j^{-}$ are $\le o(\e_j^{n})$. 
Since $(\tau_j^+,p_j^+, Q_j^+)\in \overline{\mathcal P}^{2+} u(t_{\e_j},x_{\e_j})$ 
we deduce from the viscosity differential inequality satisfied by $u$ 
that $H_j^+$ is positive definite and that the product of its $n$
eigenvalues is $\ge c>0$ uniformly in $j$ (see \cite[Theorem 2.5]{EGZ14}). In particular its 
smallest eigenvalue is $\ge c\e_j^{n-1}$. The relation $ H_{j}^+ + o(\e_j^{n})\le H_j^{-}$
forces $H_j^{-} >0$ for $j > 1 $ large enough and
$ \text{det} \, H_j^+ \leq \text{det} \, H_j^-  + o(\e_j)$.

From the viscosity differential inequalities satisfied by $u$ and $v$, we deduce that
 \begin{eqnarray*} 
 && e^{\tau_j^-  +  \frac{\delta}{(T -t_{\e_j})^2} + \tilde F (t_{\e_j},x_{\e_j}, u (t_{\e_j},x_{\e_j}))} \tilde \mu (t_{\e_j},x_{\e_j}) \\
 &&\leq    \text{det} \, H_j^+ 
 \leq  \text{det} H_j^-  + o(\e_j) \\
 && \leq  e^{\tau_j^- + \tilde G (t_{\e_j},y_{\e_j}, v (t_{\e_j},y_{\e_j}))} \tilde \nu (t_{\e_j},y_{\e_j}) + o(\e_j).
\end{eqnarray*} 

Therefore for $j >$ large enough, we get
$$
 e^{\frac{\delta}{(T -t_{\e_j})^2} + \tilde F (t_{\e_j},x_{\e_j}, u (t_{\e_j},x_{\e_j}))} \tilde \mu (t_{\e_j},x_{\e_j}) 
  \leq  e^{ \tilde G (t_{\e_j},y_{\e_j}, v (t_{\e_j},y_{\e_j}))} \tilde \nu (t_{\e_j},y_{\e_j}) + e^{- C} o(\e_j).
$$ 
 
Then letting $j \to + \infty$, we infer the following (see \cite[Lemma 3.1]{EGZ14})
 \begin{equation} \label{eq:FI1}
 e^{\frac{\delta}{(T -\hat t)^2} +  \tilde F (\hat t,\hat x, \tilde u (\hat t,\hat x))}  \tilde \mu (\hat t,\hat x)  \leq e^{ \tilde G (\hat t,\hat x, v (\hat t,\hat x))} \tilde \nu (\hat t,\hat x).
\end{equation}

Back to $\f$ and $\p$ we then get  the required inequality (\ref{eq:FIneq0}).

If $\nu = \mu > 0$ and $G = F$, then we get
$$
 \frac{\delta}{(T -\hat t)^2} + F (\hat t,\hat x, \tilde \f (\hat t,\hat x))  <  F  (\hat t,\hat x, \p (\hat t,\hat x))
$$
 Since $ F$ is non decreasing in the last variable, it follows that
 $$
 \tilde \f (\hat t,\hat x)  \leq \p (\hat t,\hat x).
 $$
 
 Taking into acount the inequality (\ref{eq:Case1}), we conclude that
 $$
 \f (t,x) - \p (t,x) - \frac{\delta}{T - t} \leq \max_X (\f (0,x) - \p (0,x)_+ - \frac{\delta}{T}.
 $$
 Letting $\delta \to 0$ we obtain the theorem.
 \end{proof}
  
\smallskip

\begin{proof} [{Proof of corollary~\ref{Cor1}}]

We first establish a more general estimate. 

Let $\mu (t,x) \geq  0$ and $\nu (t,x) \geq 0$ be  two non negative continuous  volume forms  on $X_T$  and $F,G : \R^+  \times X \times \R \longrightarrow \R$ two continuous functions. 

Assume that $\f$  is a bounded subsolution to the parabolic complex 
Monge-Amp\`ere equation (\ref{eq:PCMA}) associated to $(\omega, F,\mu)$  and $\p$ is a bounded  supersolution to the parabolic complex Monge-Amp\`ere equation (\ref{eq:PCMA}) associated to $(\omega,G,\mu)$  in $X_T$.

Let $\theta$ as in (\ref{Condition}.b)  and $\rho < 0$ be a bounded $\theta$-psh  function in $X$ satisfying 
$(\theta + dd^c \rho)^n \geq  \lambda_0 > 0$ for a fixed positive volume form $\lambda_0$ on $X$ \cite{EGZ09}.

 Fix $\e \in ]0,1[$ and set 
$$
 \f^\e (t,x) := (1 - \e) \f(t,x)  + \e \rho -  A t,
$$
where $A = A (\e)> 0$ is a constant to be chosen later. Then 
$$
(\omega + dd^c \f^\e)^n \geq (1 - \e)^n (\omega + dd^c \f)^n + \e^n \lambda_0.
$$
Since $\f^\e \leq \f + M \e$, where $M$ is a bound for the $L^{\infty}$-norm of $\f$ and $\partial_t \f \geq - C$, it follows 
that $\partial_t \f^\e \leq \partial_t \f + C \e$, in the sense of viscosity and then
\begin{eqnarray*}
e^{\partial_t \f^\e +  F (t,x,\f^\e)} \mu (t,x) & \leq & e^{\partial_t \f - A +  \e C + F (t,x,\f) + M \kappa \e} \mu (t,x),\\
& \leq & (1 - \e)^n (\omega + dd^c \f)^n,
\end{eqnarray*}
if we choose $A :=  \e C + M \kappa \e - n \log (1 - \e)$. Here, we used (\ref{Condition}.c) to introduce $\kappa$  a uniform Lipschitz 
constant for $F$ with respect to the variable $r$.

Therefore 
$$
(\omega + dd^c \f^\e)^n \geq e^{\partial_t \f^\e +  
F (t,x,\f^\e)}\mu (t,x) + \e^n \lambda_0.
$$

Observe that since $\partial_t \f \leq C$, 
$$
\partial_t \f^\e  +  F (t,x,\f^\e) \leq C (1-\e) + B_0
$$
where $B_0>0$ exists thanks to (\ref{Condition}.d), 
and choosing $\eta \medskip:=\e^n e^{ - C (1-\e) - B_0}$ we obtain 
$$
(\omega + dd^c \f^\e)^n \geq e^{\partial_t \f^\e +  
F (t,x,\f^\e)}(\mu (t,x) + \eta \lambda_0).
$$

Thus $\f^\e$ is a subsolution to the  parabolic equation associated to $(\omega, F,\mu (t,x) + \eta \lambda_0)$.
Since the volume form  $\mu (t,x) + \eta \lambda_0$ is positive, we can apply   the inequality (\ref{eq:FIneq0}) to conclude that

\begin{equation} \label{eq:FIneq3}
e^{ \frac{\delta}{(T -\hat t_\e)^2} +  F (\hat t_\e,\hat x_\e,  \tilde\f^\e (\hat t_\e,\hat x_\e))} (\mu (\hat t_\e,\hat x_\e) + \eta \lambda_0)  \leq e^{ G (\hat t_\e,\hat x_\e, \p (\hat t_\e,\hat x_\e))} \nu (\hat t_\e,\hat x_\e),
\end{equation}
when there exists a point   $(\hat t_\e, \hat x_\e) \in ]0,T[ \times X$ where $\tilde \f ^\e- \p$ achieves its maximum on $X_T$.
In particular,  $\nu(\hat t_\e, \hat x_\e)>0$.

If moreover $\mu = \nu$, it follows from (\ref{eq:FIneq3}) that $\mu (\hat t,\hat x) > 0$ and then
\begin{equation} \label{eq:FIneq4}
 \frac{\delta}{(T -\hat t)^2} +  F (\hat t,\hat x,  \tilde \f^\e (\hat t,\hat x))   \leq  G (\hat t,\hat x, \p (\hat t,\hat x)).
\end{equation}

 If moreover $F = G$, we conclude as before that
$$
 \tilde \f^\e - \p \leq \max_X (\tilde \f^\e_0 - \p_0)_+ \leq \max_X ( \f_0 - \p_0)_+.
$$
Letting $\delta \to 0$ and then $\e \to 0$, we obtain the required inequality 
$\f - \p \leq  \max_X (\f_0 - \p_0)_+$. 
\end{proof}

 \smallskip

\begin{proof} [{Proof of corollary~\ref{Cor2}}]
Here we assume that the forms $\omega$ do not depend on the time variable. We will try and do the proof when $\mu$ depends on $t$ in order to 
stress the role of the hypothesis that $\mu$ is time independant. 

We are going to regularize in the time variable to reduce to the previous case.
  Let $(\f^k)$ be the upper Lipschitz regularization of $\f$ and $(\p_k)$ the lower Lipschitz regularization of $\p$ in the variable $t$  \cite[Lemma 2.5]{EGZ14}.
 Recall that
 \begin{eqnarray*}
 \f^k (t,x) & := & \sup \{\f (s,x) - k \vert s -t\vert , s \in [0,T[\},\\
 \p_k (t,x) & := & \inf \{\p (s,x) + k \vert s -t\vert , s \in [0,T[\},
 \end{eqnarray*}

Since $\omega$ do not depend on the time variable, if follows that for each $k > 1$, $\f^k$ is a subsolution to the parabolic equation associated to
$(\omega,F_k,\mu_k)$ and $\p_k$ is a supersolution to the parabolic equation associated to $(\omega,F^k,\mu^k)$ (see \cite{EGZ14}, Lemma 2.5). 
 Recall that
 \begin{eqnarray*}
 F^k (t,x,r) &:= &\sup \{ F (s,x,r) - k \vert s - t\vert ; s \in [0,T[, \  |s-t|\le \alpha/k\},\\
  F_k (t,x,r) &:= &\inf \{ F (s,x,r) + k \vert s - t\vert ; s \in [0,T[, \  |s-t|\le \alpha/k\},\\
  \mu_k(t,x)& :=& \inf \{ \mu(s,x); |s-t|\le \alpha/k\},\\
  \mu^k(t,x)&:=& \sup \{ \mu(s,x); |s-t|\le \alpha/k\}\cdot
\end{eqnarray*}
for some $\alpha>0$.

As in the proof of Corollary~\ref{Cor1} define, for  $0<\e<1$, $\f^{k,\e} (t,x) := (1 - \e) \f^k (t,x) + \e \rho (x)-A_k(\epsilon)t$
and $\tilde \f^{k,\epsilon}:=\f^{k,\e}-\frac{\delta}{T-t}$.
 Then we can apply  the inequality (\ref{eq:FIneq3}) in the proof of Corollary~\ref{Cor1}  to deduce that:

\begin{equation} 
 e^{\frac{\delta}{T^2} +  F_k (\hat t,\hat x, \tilde \f^{k,\e} (\hat t,\hat x))} (\mu_k+\eta \lambda_0)   \leq e^{ F^k (\hat t,\hat x, \p_k (\hat t,\hat x))}\mu^k,
\end{equation}
where $(\hat t, \hat x)= (\hat t_{\delta, k,\e}, \hat x_{\delta, k,\e}) \in ]0,T[ \times X$ is a point where $\tilde \f^{k,\e} - \p_k$ achieves its maximum on $X_T$.

By construction $ \hat t \le T_{\delta}<T$ where $T_{\delta}$ does not depend on $k,\epsilon$. 
Since $F_k,  F^k \to F$  locally uniformly and $\mu_k = \mu^k = \mu$
\footnote{Here we use the hypothesis that $\mu$ does not depend on $t$. Without this hypothesis an error term
$\log(\frac{\mu^k}{\mu_k+\eta_k(\e) \lambda_0})$ appears that may diverge to $+\infty$ when $\epsilon\to 0$, $k$ being kept fixed, 
say if $\mu( t,x)=0$ for $t\le \hat t$ but $\mu(t,x)>0$ for $t>\hat t$.},  for $k$ large enough we get
$$
 \frac{\delta}{2 T^2} + F (\hat t,\hat x, \tilde \f^{k,\e} (\hat t,\hat x)) \leq  F (\hat t,\hat x, \p_k (\hat t,\hat x)).
$$
 Since $F$ is non decreasing in the last variable, it follows that for $k > 1$ large enough and for all $0<\epsilon <1$,
 $$
 \tilde \f^{k,\e} (\hat t,\hat x) < \p_k (\hat t,\hat x).
 $$
 Therefore we get
 $$
 \max_{X_T}  (\tilde \f^{k,\e} - \p_k) \leq \max_{X}  (\tilde \f^{k,\e} (0,\cdot) - \p_k (0,\cdot))_+\leq \max_{X}  (\tilde \f^{k} (0,\cdot) - \p_k (0,\cdot))_+
 $$
 
 First let $\epsilon \to 0$. It follows that:
 
  $$
 \max_{X_T}  (\tilde \f^{k} - \p_k) \leq \max_{X}  (\tilde \f^{k} (0,\cdot) - \p_k (0,\cdot))_+
 $$
 
 Now we let $k \to + \infty$ and use Dini-Cartan's lemma to conclude that $
 \max_{X_T}  (\tilde \f - \p) \leq \max_{X}  (\tilde \f (0,\cdot) - \p (0,\cdot))_+
 $, which implies the required estimate as $\delta \to 0$.
  \end{proof}
 
 \smallskip

\begin{proof} [{Proof of corollary~\ref{Cor3}}]
 Now assume that $\omega$ depends on $t$. Then from the proof of  Lemma 2.5 in \cite{EGZ14}, we see that for any $(t_0,x_0) \in ]0,T[ \times X$ there exists $t_0^* \in ]t_0 - \alpha \slash k, t_0 + \alpha\slash k[$ such that 
 $\f^k$ satisfies the viscosity inequality
 $$
 (\omega (t_0^*,x_0) + dd^c \f^k (t_0,x_0))^n \geq e^{\partial_t \f^k (t_0,x_0) + F (t_0^*, x_0, \f^k (t_0,x_0))} \mu (x_0),
 $$
 where $\alpha > 0$ is a constant.
 
 Now assume  that $t \longmapsto \omega (t,\cdot)$ is non decreasing.
 Then for $k > 1$ large enough, $(\omega (t_0^*,x_0) \leq (\omega (t_0 + \alpha\slash k,x_0)$ and then the function $u^k (t,x) := \f^k (t - \alpha\slash k,x)$ is
 a a subsolution to the parabolic equation associated to $(\omega,\hat F_{k},\mu)$ in $]\alpha\slash k,T[ \times X$, where 
 $$
 \hat F_{k} (t,x,r) := F_{k} (t-\alpha\slash k,x,r).
 $$
 In the same way, we see that the function $v_k := \p_k (t + \alpha\slash k,x)$ is a supersolution to the parabolic equation associated $(\omega,\hat F^{k},\mu)$ in $]0,T[ \times X$, where 
 $$
 \hat  F^{k} (t,x,r) := F^k (t+ \alpha\slash k,x,r).
 $$
 
 Then one modiifes easily  the proof of Corollary~\ref{Cor2}, with $u^k$ replacing $\phi^k$ and $u_k$ replacing $\psi_k$.
  
  It is clear that the same argument works in the non increasing case.
 \end{proof}
 
 \smallskip
 
\begin{proof} [{Proof of corollary~\ref{Cor4}}]
By Lemma 2.5 in  \cite{EGZ14},  $\f^k$ is a subsolution of the equation associated to $((1+E(\alpha/k))\omega_t, F_k, \mu)$ whereas $\p_k$ is a supersolution of the 
equation associated to $((1-E(\alpha/k))\omega_t, F_k, \mu)$ with $\alpha>0$ as above. Hence $\f_{\star}^{k}=\frac{\f^k}{1+E(\alpha/k)}$ is a subsolution 
of the equation  to  $(\omega_t, F_k-\log(1+E(\alpha/k)), \mu)$ and $\p_{{\star}k}$ is a supersolution of the equation associated to  $(\omega_t, F^k+\log(1+E(\alpha/k)), \mu)$.
We can now argue exactly as in the proof of Corollary~\ref{Cor2}, with $\f_{\star}^k$ replacing $\f^k$ and $\psi_{{\star}k}$ replacing $\psi_k$.
 \end{proof}

\begin{rem} \label{twist-PMA}
Renormalization in the time variable leads to twisted parabolic complex Monge-Amp\` ere equation equations of the type
 \begin{equation} \label{eq:twist-PMA}
  e^{h (t) \partial \f_t + F (t,\cdot,\f)} \mu  - (\omega_t +dd^c \f_t)^n = 0
 \end{equation}
 in $[0,T[ \times X$, 
 where $h :[0, T[ \longrightarrow ]0,+ \infty[$ is a continuous positive function.
   
   The comparison principle Theorem~\ref{SCP} holds for the twisted parabolic complex Monge-Amp\` ere equation (\ref{eq:twist-PMA}) as in the local case (see \cite{EGZ14}).
  \end{rem}

\section{Barrier constructions}

Let $X$ be a compact K\"ahler  manifold of dimension 
$n$ and $\omega_0$ is  semipositive closed $(1,1)$ form with positive volume. 
We consider in this section the Cauchy problem  on $X_T$
\begin{equation} \label{eq:CPMA}
\left\{
\begin{array}{ll}
 e^{ \partial_t \f + \al \ \f} \mu - (\omega_t+dd^c \f_t)^n = 0 \\
  \f (0,x) = \f_0 (x), \, \, \, \, (0, x) \in  \{0\} \times X,
\end{array}
\right.
\end{equation} 
where $\f_0$ is a given continuous $\omega_0$-plurisubharmonic function on $X$ and $\alpha \in \R^+$.

The Cauchy problem does not necessarily admit a solution when $\mu$ vanishes identically on an open set
(see Proposition \ref{exa:vanishing}). We first treat the case when $\mu>0$ is positive, and then allow 
$\mu$ to vanish along pluripolar sets. This latter setting contains as a particular case the K\"ahler-Ricci flow on varieties with canonical singularities.

\smallskip

 We will mainly focus on the case $\al =0$. The case $\al > 0$ is actually easier and can be reduced to the previous one by a change of time variable. We also need to impose some uniformity in the positivity properties of
the forms we are dealing with:

{\it We  assume in the whole section that $X$ is a compact K\"ahler  manifold of dimension 
$n$  and 
 there exists a closed real $(1,1)$-form $\theta$ on $X$ whose cohomology class is 
 semi-positive and a K\"ahler form  $\Theta$ such that for all $0 \leq t \leq T$, the background continuous family of closed $(1,1)$-forms satisfies:
\begin{equation} \label{eq:H}
 \theta \leq \omega_t \leq \Theta.
\end{equation} 
}

\subsection{Existence of sub/super solutions}

\begin{lem} \label{Sub-Super} 
The  Cauchy problem (\ref{eq:CPMA})  admits a continuous subsolution $\underline u$, Lipschitz in the variable $t$. 

 Assume $\mu (t,x)\geq f_0 (x) d V$, where $f_0 \geq 0$ is a continuous density such that
 $$
 \int_X f_0 \, d V > 0.  \leqno (\dagger)
 $$
  Then, there exists a  continuous  supersolution $\overline v$, Lipschitz in the variable $t$.
 
 Moreover we can choose these so that $\underline u \leq \overline v$ in $ [0,T[ \times X$.
 \end{lem}

 \begin{proof} 
 By \cite{EGZ09}, there exists a continuous $\theta$-psh function $\rho_1$ in $X$ such that $(\theta + dd^ c \rho_1)^ n = c_1 d V$ in the weak sense on $X$, where $c_1$ is a normalizing constant. We can normalize $\rho_1$ so that $\rho_1 \leq \f_0$ in $X$.  Define for $C_1 >0$, the function  
$$
\underline u := - C_1 t + \rho_1 (x).
$$

 Then, by Lemma \ref{lem:pluripotvisc}, if $C_1 >> 1$ is choosen so large that $e^{- C_1} \sup_{X_T} \mu \leq c_1 d V$, the function $\underline u$  is a subsolution to the  Cauchy problem (\ref{eq:CPMA}). 

In the same way we construct a supersolution. 
Since $f_0 \geq 0$ is a bounded upper semi-continuous function on $X$ and $\int_X f_0 d V > 0$,  there exists  a continuous $\Theta$-psh $\rho_2$ satisfying
$$
(\Theta + dd^ c \rho_2)^ n = c_2 f_0 (x) \lambda_0
$$
 in the weak sense on $X$, where $c_2$ is a normalizing constant (by \cite{Kol98,EGZ11}). 
We normalize $\rho_2$ so that $\rho_2 \geq \f_0$ in $X$. 
Consider the function 
 $$
 \overline v := + C_2 t + \rho_2,
 $$
 where  $C_2 > - \log c_2$ is a positive constant. 

 Lemma \ref{lem:pluripotvisc} implies then  that $ \overline v$ is also a supersolution to the parabolic complex Monge-Amp\` ere equation (\ref{eq:CPMA}).
 Since $\overline v \geq \f_0$ in $X$ we obtain a continuous supersolution to the Cauchy problem (\ref{eq:CPMA}).
 \end{proof}

 \begin{coro}
  Assume either $\mu>0$ or the hypotheses of Corollaries \ref{Cor3} or \ref{Cor4} are satisfied in addition
  to those of lemma \ref{Sub-Super}$(\dagger)$. Then the Cauchy   problem (\ref{eq:CPMA})
  is admissible. 
 \end{coro}

Fix $\underline u, \overline v$ a  subsolution and a supersolution of the Cauchy problem  (\ref{eq:CPMA}). We are now in the position to apply Proposition \ref{perron}. The natural candidate to be a solution is the upper envelope of subsolutions  
 \begin{equation} \label{eq:MaxSub}
 \f := \sup \{ u \; | \; u \in \mathcal S, \underline u \leq \p \leq  \overline v\},
 \end{equation}
 where $\mathcal S$ denotes the family of all subsolutions to the Cauchy problem (\ref{eq:CPMA}). 
We let $\f^*$ denote the upper semi-continuous regularization of $\f$ and $\f_*$ denote its
lower semi-continuous regularization. It follows that:

 \begin{coro} \label{lem:Perron}

Assume the hypotheses of Corollaries \ref{Cor3} or \ref{Cor4} are satisfied in addition
  to those of lemma \ref{Sub-Super}$(\dagger)$. 
  
The upper semi-continuous regularization $\f^*$ is a discontinuous viscosity solution to the underlying 
parabolic Monge-Amp\`ere equation  in $]0, T[ \times X$.

The lower semi-continuous regularization $\f_*$ is thus a supersolution 
to the parabolic Monge-Amp\`ere equation  in $]0, T[ \times X$ and they satisfy
 then for all  $(t,x) \in ]0,T[ \times X$, 
 \begin{equation} \label{eq:Perron}
 \f^* (t,x)  - \f_* (t,x) \leq \max_{x \in X} (\f^* (0,x)  - \f_* (0,x)).
 \end{equation}
 \end{coro}

If we could make sure that $\f^* \leq  \f_0 \leq \f_*$ on the parabolic boundary $\{0\} \times X$, it would follow
from the inequality (\ref{eq:Perron}) that $\f^* =\f_* =  \f$ is a unique viscosity solution of the Cauchy problem.
Establishing this classically requires the construction of barriers at each boundary point in $\{0\} \times X$.

\subsection{Existence of barriers}
 
   \begin{defi}  
Fix $(0,x_0) \in \{0\} \times X$ and  $\e \geq 0$.

 1.  An upper semi-continuous function $u : X_T \longrightarrow \R$ is an $\e$-subbarrier  to the Cauchy problem (\ref{eq:CPMA}) at the boundary point $(0,x_0)$, if 
\begin{itemize}
\item $u$ is a subsolution  to the  Monge-Amp\`ere flow (\ref{eq:CPMA}) in $]0,T[ \times X$,
\item $u (0,\cdot) \leq \f_0$ in $X$, 
\item $u_* (0,x_0) \geq \f_0 (x_0) - \e$. 
\end{itemize}
When $\e = 0$, $u$ is called a subbarrier.

 2.  A lower semi-continuous function $v : X_T \longrightarrow \R$ is an $\e$-superbarrier  to the Cauchy problem (\ref{eq:CPMA}) at the boundary point $(0,x_0)$, if 
\begin{itemize}
\item $v$ is a supersolution  to the  Monge-Amp\`ere flow (\ref{eq:CPMA}) in $]0,T[ \times X$,
\item $v (0,\cdot) \geq \f_0$ in $X$ 
\item $v^* (0,x_0) \leq \f_0 (x_0) + \e$. 
\end{itemize}
When $\e = 0$ $v$ is called a superbarrier.
 \end{defi}
 
We now investigate the existence of sub/super-barriers. 

\begin{prop} \label{prop:bar} 
\text{ }

1. Assume $\omega_0 \leq \omega_t$ and fix $\e > 0$.  
There exists a continuous function $U_\e$ in $X_T := [0,T[ \times X$, 
Lipschitz in $t$  which is an $\e$-subbarrier to the Cauchy problem (\ref{eq:CPMA}) at any point 
 $(0,x_0) \in \{0\} \times X$.

2. Assume $\mu (t,x) > 0$ in $X_T$ and fix $\e > 0$. There exists a  continuous function $V_\e$ in $X_T$, Lipschitz in $t$,   which is a $\e$-superbarrier 
to the Cauchy problem (\ref{eq:CPMA})  at any  point $(0,x_0) \in \{0\} \times X$.
\end{prop}

As the proof will show one can moreover impose that for all $(t,x) \in X_T$,
$$ 
- C_1 t + \rho_1 (x) \leq U_\e (t,x)  \leq V_\e (t,x) 
\leq C_2 t + \rho_2 (x),
$$
where $C_1, \rho_1, C_2, \rho_2$ are independent of $\e$ and given in Lemma~\ref{Sub-Super}.

\begin{proof} 
1.  By \cite{EGZ11}, since $\mu$ is continuous, there exists  $w_0$ a continuous $\theta$-psh function on $X$  such that 
$(\theta + dd^c w_0)^n \geq e^{w_0} \mu $.
 Adding a negative constant we can always assume that $w_0 \leq \f_0$ in $X$.

Fix $\e > 0$, $\eta=\eta_\e> 0$, $C=C_\e > 0$ (to be chosen below) and set
$$
u (t,x) := (1- \eta) \f_0 (x) + \eta w_0 (x) - C t, \, \, (t,x) \in X_T.
$$ 
This is a continuous function in $X_T$ such that for any $t \in [0,T[$. Since $\omega_0 \leq \omega_t$, $u_t$ is 
$\omega_t$-psh 
in the space variable $x \in X$ and satisfies the differential inequalities
$$
(\omega_t+dd^c u_t)^n \geq \eta^n (\theta+dd^c w_0)^n \geq \eta^ n e^{w_0} \mu  \text{ on } \, X,
$$
while  $\partial_t u = - C$ in $X_T$. We choose $C = C (\eta) > 1$ large enough so that 
$\eta^n e^{w_0} \geq e^{- C}$, hence for each $t \in ]0,+ T[ $ we have 
$$
(\omega_t + dd^c u_t)^n \geq e^{\partial_t u (t,\cdot)} \mu
$$
  
Note that $u _0 = \f_0 + \eta (w_0 - \f_0) \leq \f_0$ in $X$.
We can choose $\eta > 0$ so small that $\eta \sup_X (\f_0 - w_0) \leq \e$ and Lemma \ref{lem:pluripotvisc} enables to conclude that $u$ 
is an $\e-$subbarrier for the Cauchy problem (\ref{eq:CPMA}) at any point $(0,x_0)$. 

We can moreover use Lemma~\ref{Sub-Super} to find a bounded subsolution $- C_1 t + \rho_1$  to the Cauchy problem (\ref{eq:CPMA}) which is independent of $\e$. Set for $(t,x) \in X_T$,
$$
U_\e (t,x) := \sup\{ u (t,x) , -C_1 t + \rho_1\}.
$$
The function $U$ is also an $\e$-subbarrier to the Cauchy problem (\ref{eq:CPMA}) at any boundary point $(0,x_0) \in \{0\} \times X$.
\smallskip

2. Constructing superbarriers.
 Fix $\e > 0$. Since $\Theta$ is K\"ahler and $\f_0$ is in particular a $\Theta$-psh function in $X$ (recall that $\omega_0 \leq \Theta$), there exists a  $C^{\infty}$-smooth  $\Theta$-psh function $\tilde \f_0$ in $X$  such that $\f_0 \leq \tilde \f_0 \leq \f_0 + \e$ in $X$ (see \cite{Dem92}, \cite {BK07}). 
Thus there is a constant $C > 0$ such that 
$$
(\Theta+dd^c \tilde \f_0)^n \leq e^C \mu
$$ 
pointwise on $ X$, as  we are assuming $ \mu > 0$.

Set $v (t,x) := \tilde \f_0 (x) + C t$ in $X_T$ and observe that
$$
(\Theta+dd^c v_t)^n = (\Theta + dd^c \tilde \f_0)^n \leq e^C \mu \leq e^{\partial_t v} \mu.
$$ 
Since $\omega_t \leq \Theta$ we infer  that $v_t$ also satisfies,  in the viscosity sense:
$$
(\omega_t+dd^c v_t)^n \leq e^C \mu.
$$ 

Therefore  $v$ is a continuous $\e$-superbarrier to the Cauchy problem (\ref{eq:CPMA}) at any boundary point  in $\{0\} \times X$.

Using Lemma~\ref{Sub-Super} and the condition $\Theta\ge \omega_t$, we moreover obtain a supersolution $\rho_2 +  C_2 t$ to the Cauchy problem (\ref{eq:CPMA}) and set for $(t,x) \in X_T$,
 $$
 V_\e (t,x) := \inf \{v (t,x) , \rho_2 (x) + C_2 t \}\cdot 
 $$
This $V$ is an $\e$-superbarrier to the Cauchy problem (\ref{eq:CPMA}) at any boundary point $(0,x_0) \in \{0\} \times X$.
\end{proof}

\begin{rem} \label{barriers}
\text{ }

1. If the Cauchy data $\f_0$ is  a continuous $\theta$-psh function on $X$ satisfying 
$(\theta +dd^c \f_0)^n \geq e^{\f_0} \mu$, then we can take $w_0 = \f_0$ in the above construction of  subbarriers. The corresponding function $U$ is then a bounded  continuous subsolution, 
which is uniformly Lipschitz in $t$ and satisfies $U (0,\cdot) = \f_0$, i.e. $U$ is a subbarrier to the Cauchy problem (\ref{eq:CPMA}).

 2. If the Cauchy data $\f_0$ is a continuous $\Theta$-psh function on $X$  such that 
$(\Theta+dd^c \f_0)^n$ has an $L^{\infty}$-density, then we can take $\tilde \f_0 = \f_0$ and $\e = 0$ in the above construction of superbarriers. We thus obtain a bounded continuous supersolution $V$ which is uniformly Lipschitz in $t$ and such that  $V (0,\cdot) = \f_0$ in $X$, i.e. $V$ is a superbarrier to the Cauchy problem (\ref{eq:CPMA}).
\end{rem}

\subsection{Non negative densities}

We explain in this section a non existence result:  when $\mu$ vanishes on an open set, 
there is no solution unless the initial data has special properties.

\begin{prop} \label{exa:vanishing}
Assume that $\mu = f d V,$ where $f \geq 0$  vanishes identically on 
$D \times [0,\delta]$, where
 $D \subset X$ is open. 
 
 If the initial data $\f_0$ is not a maximal $\omega$-psh function in $D$, then the Cauchy problem (\ref{eq:CPMA}) has no viscosity solution.
  \end{prop}

  Recall that a continuous $\omega$-psh function $u$ is {\it maximal}  in $D$ if it satisfies the 
homogeneous complex Monge-Amp\`ere equation $(\omega + dd^ c u)^ n = 0$  there.

\begin{proof}  
Assume that the Cauchy problem (\ref{eq:CPMA})  with initial data $\f_0$ has a solution 
$\f$ in $[0,\delta] \times X$. Since $\mu = 0$ in $[0,\delta] \times D$,
 it follows that $\f$ is a solution to the degenerate parabolic equation $(\omega_t + dd^c \f_t)^n = 0$
in $D \times [0,\delta]$.
 
  We claim that for almost every $t > 0$, the function $\f_t$ is a continuous 
$\omega_t$-psh function on $X$, which is a viscosity solution of the elliptic  equation
$$
(\omega_t + dd^c \f_t)^n= 0.
$$

 This is clear if $ \f$ is  a classical solution. To treat the general case we use inf convolution to approximate $\f$ by an increasing sequence $(\f_j)$ of semi-concave functions which satify 
 the same equation on a slightly smaller domain  that we still denote by $[0,\delta] \times D$
for simplicity. The functions $\f_j$ admit a $(1,2)$-Taylor expansion almost everywhere, hence for a.e.  $(t,x)$, 
$$
(\omega + dd^c \f_j (t,x))^n= 0.
$$
Fixing one such $t$, it follows that for almost every $x$, 
$$
(\omega + dd^c \f_j (t,\cdot))^n= 0.
$$
It follows that the latter actually holds everywhere in $D$ in the viscosity sense (see \cite{EGZ14}).

 Since $\f_j$ increases to $\f $, it follows from the continuity of the complex Monge-Amp\`ere operator along montone sequences that for almost every $t$  the function  $\f_t$ satisfies 
 $(\omega + dd^c \f (t,\cdot))^n= 0$  in $D$.

 Note finally that $\f_t \to \f_0$ uniformly, hence $\f_0$ is maximal in $D$. 
 \end{proof}

\subsection{Canonical vanishing: existence of solutions}

We now restrict our attention to semi-positive measures 
$$
\mu(x,t)=e^{u (x)} f(x,t) dV(x),
$$
where $f>0$ is a positive continuous density and $u$ is  quasi-plurisubharmonic function that is
  {\it exponentially continuous} (i.e. such that $e^u$ is continuous).
The measure $\mu$ is thus allowed to vanish only along the closed pluripolar set
$(u=-\infty)$, in a time independent fashion.

\begin{lem} \label{SuperBar}
 For any $\e > 0$ there exists a lower semi-continuous function $ w : [0,T[ \times X \longrightarrow \R$, which is an $\e$-superbarrier  to the Cauchy problem (\ref{eq:CPMA}) at any boundary point $(0,x_0)$ with $u (x_0) > - \infty$.
\end{lem}
\begin{proof}
We can assume without loss of generality that $u \leq 0$ is a $\Theta$-psh function on $X$. Fix $\e > 0$. 
From the approximation theorem of Demailly (see \cite{Dem92}, \cite{BK07}), it follows  
that there exists a smooth $\Theta$-psh function $\tilde \f_0$ in $X$ such that
$ \f_0 \leq \tilde \f_0 \leq \f_0 + \e$ in $X$.  Set
$$
v (t,x) := \tilde \f_0 - t u + C t, \, \, (t,x) \in [0,T[ \times \Omega,
$$
where $\Omega := \{ x \in X \, | \, u (x) > - \infty\}$ is open and $C > 0$ is a constant to be chosen later. 
Observe that $v$ is continuous in $[0,T[ \times \Omega$ and satisfies 
$$ 
\Theta + dd^c v_t =  2 \Theta + dd^c \tilde \f_0 - t (dd^c u + \Theta)  + (t- 1) \Theta, 
$$
in the sense of currents in $\Omega$.
Since $dd^c u + \Theta \geq 0$, for $0 < t \leq T$, we have 
$$
\Theta + dd^c v_t \leq  2 \Theta + dd^c \tilde \f_0
$$
 in the sense of currents in $ \Omega$.

We choose $C > 1$ so big that  $(2 \Theta + dd^c \tilde \f_0)^n \leq e^C d V$. 
This we can do since $2 \Theta + dd^c \tilde \f_0$ is a smooth positive form on $X$. 

Note that $e^{\partial_t v } = e^{C-u}$ thus it follows from Lemma~\ref{lemVSUP} that $v$ satisfies the viscosity parabolic differential inequality $(\Theta + dd^c v_t)^n \leq e^{\partial_t v } \mu$ in $[0,T[ \times \Omega$.
As $\omega_t \leq \Theta$, Lemma~\ref{lemVSUP}  also implies that $v$ is a supersolution to the parabolic Monge-Amp\`ere equation  (\ref{eq:CPMA}) in $[0,T[ \times \Omega$.

On the other hand we know that there exists a (continuous) supersolution $\overline v$ to the parabolic Monge-Amp\`ere equation (\ref{eq:CPMA}) in $\R^+ \times X$ such that $\overline v_0 \geq \f_0$ in $X$.
The function $w := \inf \{v, \overline v\}$  a bounded, lower semi-continuous in $[0,T[ \times X$ and continuous in $[0,T[ \times \Omega$.  It can thus be extend as a lower semi-continuous function on $[0,T[ \times X$ by setting  
$$
w (0,x_0) := \inf \{ \tilde \f_0 (x_0) , \overline v_0 (x_0)\}
$$
for  any point $(0,x_0) $ with $u (x_0) = - \infty$. 
We let the reader check that this extension, which we still denote by $w$, is a supersolution to the parabolic Monge-Amp\`ere equation (\ref{eq:CPMA}) in $]0,T[ \times X$ such that $\f_0 \leq w_0 \leq \f_0 + \e$ in $X$.

Fix a point $x_0 \in X$ such that $u (x_0) > - \infty$.  Then $w^*(0,x_0) = w (0,x_0) \leq \f_0 (x_0) + \e$,
hence $w$ is an $\e$-superbarrier at such a point.
\end{proof}
In the proof above, we have used the following technical result:

\begin{lem} \label{lemVSUP} 
Let $\mu \geq 0$ be a continuous volume form on some domain $D$.
Let $\p$ be a bounded lower semi-continuous function in $D \subset X$ and $\rho$ a $C^ 2$-smooth function in $D$ such that $dd^ c \p \leq dd^ c \rho$ in the sense of currents. 
Then $(dd^c \p)^n \leq (dd^ c \rho)_+^n$ in the viscosity sense in $D$. 

If  $\Theta_1$ and $\Theta_2$ are smooth closed real $(1,1)$-forms in $X$ such that
$\Theta_1 \leq \Theta_2$ 
and  $(\Theta_2 + dd^c \p)^n \leq \mu$ in the viscosity sense, 
then $(\Theta_1 + dd^ c \p)^n \leq \mu$   in the viscosity sense.
\end{lem}

Recall that  $(dd^ c \rho)_+$ is the $(1,1)$-form defined pointwise by   
$(dd^ c \rho)_+ (x_0) := dd^ c \rho (x_0)$ if $dd^ c \rho (x_0) \geq 0$ and $0$ otherwise.

\begin{proof}
If $q$ a $C^2$ lower test function for $\p$ at a point $x_0 \in D$, i.e. $q \leq_{x_0} \p$,
 then $\rho - \p \leq_{x_0} \rho - q$.
Since $dd^ c \p \leq dd^ c \rho$, it follows that $\rho - \p$ is plurisubharmonic in $D$. 
Hence  $dd^c (\rho - q) (x_0) \geq 0$, i.e.
$dd^c \rho (x_0) \geq dd^c q (x_0)$.
If $dd^c q (x_0) \geq 0$ it follows that $dd^c \rho (x_0) \geq 0$ and $(dd^c q (x_0))^n \leq (dd^c \rho (x_0))^n$.
This proves the first statement. 

The proof of the second statement goes along the same lines.
\end{proof}

 \begin{defi}\label{defi:oregular2} Say $t\mapsto \omega_t$ is very regular if it is regular in the sense of definition \ref{defi:oregular} and there exists $\eta >0$, a function of class $C^1$
$\epsilon:[0,T[\to [0,1-\eta]$ such that $\epsilon(0)=0$ and $\omega_t\ge (1-\epsilon(t))\omega_0$.
 \end{defi}
 
 As we will see in the next section, this condition is satisfied in many geometric situations and the following result will be important for our applications.

\begin{theo}  \label{thm:can}
Assume that $\mu = e^{u} f d V$ is as above and $t \mapsto \omega_t$ is non decreasing or is very regular in the sense of Definition \ref{defi:oregular2}.
Then the maximal subsolution $\f$ constructed in Proposition~\ref{perron} is a unique viscosity solution to the Cauchy problem (\ref{eq:CPMA}).
\end{theo}

\begin{proof}   
We first assume $t \mapsto \omega_t$ is non decreasing. 

By Proposition~ \ref{prop:bar},  given $\e > 0$ there exists a continuous $\e$-subbarrier $U$ at any point $(0,x_0) \in \{0\} \times X$ i.e. 
$U \leq \f$ and $U (0, x_0) \geq \f_0 (x_0) - \e$.  Since $U$ is continuous, it follows that $U \leq \f_*$ in $\R^+ \times X$, hence $\f_* (0,x_0) \geq \f_0 (x_0)$ for any $x_0 \in X$. 
This shows that $\f_*$ is a supersolution to the Cauchy problem (\ref{eq:CPMA}).

We claim that $\f^* (0,\cdot) \leq \f_0$ in $X$. 
Indeed if we fix $\e > 0$,  by Lemma~\ref{SuperBar} there exists an  $\e$-superbarrier $w$ to the Cauchy problem (\ref{eq:CPMA}) in $[0,T[ \times X$ at any point $(0,x_0)$ with
$u (x_0) > - \infty$. 

Since $w$ is a supersolution to the Cauchy problem (\ref{eq:CPMA}) in $[0,T[ \times X$, it follows from the comparison principle and the continuity of $w$ that  
$\f  \leq w$ in  $]0,T[ \times X$. Since $w$ is continuous up to the boundary, 
$$\f^ * (0,x_0) \leq w (0,x_0) \leq \f_0 (x_0) + \e
$$
 for any $x_0 \in X$ with $u (x_0) > - \infty$. 

Therefore  $\f^* (0,\cdot) \leq \f_0$ almost everywhere in $X$, since the set $\{u = - \infty\}$  
has Lebesgue measure $0$. 
Since the slice function $\f^* (t,\cdot)$ is $\omega_t$-plurisubharmonic  for all $t>0$ (\cite{EGZ14}, Theorem 2.5), and $f^*$ is upper semicontinuous on $[0,T[\times X$ it follows that 
$\f^*_0$ is $\omega_0$-plurisubharmonic. Hence $\f^*(0,x) \leq \f_0(x) $ for all $x \in X$.

We have proven that $\phi^*_0 \le \phi_0 \le \phi^*$.
It follows therefore from Lemma~\ref{lem:Perron}  that $\f^* = \f_* = \f = \p$ in $[0,T[ \times X$ is the unique solution to the Cauchy problem (\ref{eq:CPMA}). 

Definition \ref{defi:oregular2} is an ad hoc definition whose only virtue is to allow the construction of a subbarrier in Proposition \ref{prop:bar} be carried out by:

$$
u(t,x):= (1-\eta-\epsilon(t))\phi_0(x) +\eta w(x) - Ct.
$$

The superbarrier construction is completely insensitive to this difficulty and the theorem follows. 
\end{proof}

\subsection{Comparison with the vanishing viscosity method}

We consider in this section the following $\e$-perturbation of Cauchy problem (\ref{eq:CPMA}) on $X_T$ with canonical vanishing given by a quasi-plurisubharmonic  function $w$:
\begin{equation*} 
\left\{
\begin{array}{ll}
 e^{ \partial_t \f + \al \f} e^{w} f d V - (\e \Theta+\omega_t+dd^c \f_t)^n = 0 \\
  \f (0,x) = \f_0 (x), \, \, \, \, (0, x) \in  \{0\} \times X,
\end{array}
\right.
\end{equation*}  
where $\f_0$ is a given continuous $\omega_0$-plurisubharmonic function on $X$. 

Here $\e \ge 0$ is a non negative constant and $\Theta$ is a K\"ahler form. Then, if $t\mapsto \omega_t$ is very regular, $t\mapsto \e\Theta + \omega_t$ is very regular too. 
In particular, Theorem \ref{thm:can} applies and for every $\e \ge 0$ we have a viscosity solution $\phi(\e)$ of the above $\e$-perturbed complex Monge-Amp\`ere flow. 

\begin{prop}\label{lem:vanish}
 $\phi(\e)$ converges locally uniformly to $\phi(0)$ in $\R^+ \times X$ as $\e \to 0$. 
\end{prop}

\begin{proof}
Since $\phi (\e')$ is a supersolution of  $\e$-perturbed complex Monge-Amp\`ere flow
whenever $\e' \ge \e$, the comparison principle implies that 
$$
\phi (0) \leq \phi(\e)\le \phi(\e') \text{ if } 0 \leq \e \le \e'.
$$
 Using \cite[section 6]{CIL92} (see also \cite[Lemma 1.7]{EGZ14})
we conclude with  the comparison principle for $\e$-perturbed complex Monge-Amp\`ere flows.  
\end{proof}

\begin{rem}
 One could also perturb $\mu$ to a smooth positive volume  form. 
\end{rem}

\section{Applications}

In this section we show that our hypotheses are satisfied when studying the (Normalized) K\"ahler-Ricci flow on a variety with canonical singularities.
We prove the existence and study the behaviour of the normalized K\"ahler-Ricci flow (NKRF for short) on such varieties
starting from an arbitray closed positive current with continuous potential.

\subsection{The normalized K\"ahler-Ricci  flow on varieties with canonical singularities}

Let $Y$ be an irreducible compact K\"ahler normal complex analytic space with only canonical singularities.  Let $\chi_0$ be a K\"ahler form on $Y$. 
We  study the existence of the normalized K\"ahler-Ricci flow on
$Y$,
$$
\frac{\partial \omega_t}{\partial t}=-\rm{Ric}(\omega_t)-\omega_t,
$$
starting from an  initial data $\omega_0 = \chi_0 + dd^c \phi_0$ with $\phi_0$ being a continuous potential which is plurisubharmonic with respect to the given K\"ahler form $\chi_0$ on $Y$. 
At the cohomological level, this yields
a first order ODE showing that the cohomology class of $\omega_t$ evolves as
$$
\{ \omega_t \} =e^{-t} \{ \omega_0 \} +(1-e^{-t} ) K_Y.
$$

We thus  defined  by
$$
T_{max} :=\sup\{ t>0,  \ e^{-t}\{\omega_0\}+(1-e^{-t}) K_{Y} \in \mathcal{K}(Y) \}
$$
the maximal time of existence of the flow. 

Recall that given a K\"ahler class on $Y$ with a smooth positive representative $\chi_0$ and $\phi_0\in PSH(Y,\chi_0)$ 
a continuous function,  the Cauchy problem with initial data $S_0 :=  \chi_0+dd^c\phi_0$
for the normalized K\"ahler-Ricci flow  is defined after a
desingularization $\pi:X \to Y$ as the Cauchy problem with initial data $\f_0 := \pi^*\phi_0$ for the flow
$(CMAF)_{X,\omega_{NKRF}, \mu_{NKRF}, r}$ (see Definition \ref{def:nkrf}). 
 We prove the following general version of  Tian-Zhang's  existence theorem for the K\"ahler-Ricci flow:

\begin{theo}\label{thm:nkrf}
 The Cauchy problem with initial data $ S_0 := \chi_0+dd^c\phi_0$
for the normalized K\"ahler-Ricci flow on $Y$ has  a unique viscosity solution defined
on $ [0,T_{max}[ \times Y$.
\end{theo}

\begin{proof}
Fix $T < T_{max}$. Since for any $t \in [0,T]$, $e^{-t}\{\omega_0\}+(1-e^{-t}) K_{Y} \in  \mathcal{K}(Y)$, one can show that there exists a smooth family
of K\"ahler forms  $(\chi_t)_{0 \leq t\leq T}\in \mathcal{K}(Y)$ such that for any $t \in [0,T]$, $\{\chi_t\} = \{\omega_t\}$.
Observe that if $\mathcal{K}_Y$ is semi-ample then $T_{max} = + \infty$ and we can take $\chi_t := e^{-t} \chi_0 +(1-e^{-t})\chi$, where $\chi$ is a smooth semi-positive representative of the canonical class $ \mathcal{K}_Y$.

In any case  we can write $\omega_t=\chi_t + dd^c \phi_t$, where $\phi$ is a solution to the corresponding Monge-Amp\`ere flow at the level of potentials,
\begin{equation} \label{eq:PMAF}
(\chi_t  +dd^c \phi_t)^n=e^{\partial_t \phi+\phi_t} dV_Y,
\end{equation}
on $Y_T$ for some admissible volume form $dV_Y$ on $Y$,  or equivalently
$$
(\theta_t+dd^c \f_t)^n=e^{\partial_t \f +\f_t} \mu_{NKRF},
$$
on a log resolution $\pi:X \rightarrow Y$, where $\mu_{NKRF}$ is a volume form on $X$ with canonical vanishing i.e. locally $\mu_{NKRF}=\Pi_E |f_E|^{2a_E} dV_X$. 
Here we write  $\f := \pi^* \phi$ and $\theta_t := \pi^* \chi_t$.

Since $(\chi_t)_{0 \leq t\leq T}$ is a smooth family of K\"ahler forms on $Y$, it follows that the family of forms $[0,T[ \ni t \longmapsto \theta_t $ is very regular on $X$ in the sense of Definition \ref{defi:oregular2}. Therefore we can apply Theorem \ref{thm:can} to get a unique solution to the Monge-Amp\`ere flow on $X_T$ for any fixed $T < T_{max}$ starting at $\f_0$. By uniqueness all these solutions glue into a unique solution of the Monge-Amp\`ere flow on $[0,T_{max}[ \times X$ starting at $\f_0$. Pushing this solution down to $Y$ we obtain a solution to the NKRF starting at $S_0$.
\end{proof}

We have recovered by a zeroth order method one of the main results  in \cite{ST2}. 
Our viscosity solution can be identified with their weak solution thanks to 
Proposition \ref{lem:vanish}.

If $Y$ is minimal, i.e.: $K_Y$ is nef, the flow is defined up to existence time $T=+\infty$, and it is natural 
to enquire about its long-term behaviour. The sequel of this and the following section  will be mainly devoted to the 
study of this problem.

Turning briefly our attention to the case when $-K_Y$ is ample, 
it follows from Lemma \ref{chgvar} that a similar result holds when $Y$ is a $\Q$-Fano variety.
We refer the reader to \cite{BBEGZ11} for background on $\Q$-Fano varieties.
The Normalized K\"ahler-Ricci flow is here
$$
\frac{\partial \omega_t}{\partial t}=-\rm{Ric}(\omega_t)+\omega_t.
$$
 and the cohomology class is again constant (equal to $c_1(Y)$) if we start from an initial data
$S_0= \chi +dd^c \phi_0$, where $\chi$ is a K\"ahler form representing $c_1(Y)$.
The flow can be written, at the level of potentials,
$$
(\chi +dd^c \phi_t)^n=e^{\partial_t \phi-\phi_t} dV_Y
$$
for some admissible volume form $dV_Y$, or equivalently
$$
(\theta_0+dd^c \f_t)^n=e^{\partial_t \f-\f_t} \mu_{NKRF},
$$
on a log resolution $\pi:X \rightarrow Y$, where $\theta_0 := \pi^* (\chi)$ and $\mu_{NKRF}$ is a volume form with canonical vanishing i.e. locally $\mu_{NKRF} =\Pi_E |f_E|^{2a_E} dV_X$.

Theorem \ref{thm:nkrf} then
guarantees that this complex Monge-Amp\`ere flow can be started from an
arbitrary continuous $\theta_0$-psh potential $\f_0$ and exists for all times $t>0$.
The long term behavior is however much more difficult to understand on 
$\Q$-Fano varieties and is related to the (mildly) singular version of the Yau-Tian-Donaldson conjecture
(see \cite{BBEGZ11,CDS,Tian12}).

\subsection{Canonically polarized varieties}

We work in this section on a minimal model of general type,
i.e. $Y$ has canonical singularities and $K_Y$ is   big and nef (hence semi-ample by a classical result of Kawamata).
This contains in particular the case when $Y$ is a canonical model, i.e. a general type
projective algebraic variety with only canonical singularities 
such that $K_Y$ is ample  (see \cite{BCHM} for the existence of a unique canonical model in every birational class of complex projective manifolds of the general type).

\subsubsection{Starting from the canonical class} In this paragraph, we assume $K_Y$ is ample. 
If we start the normalized K\"ahler-Ricci flow from an initial data $S_0= \chi_0 +dd^c \phi_0$ whose cohomology class
$\{\chi_0\}=c_1(K_Y)$ is the canonical class, then $\{ \omega_t\} \equiv c_1(K_Y)$ is constantly
equal to the canonical class of $Y$. Thus $\omega_t=\chi_0 + dd^c \phi_t$ and the flow can be written,
at the level of potentials,
$$
(\chi_0  + dd^c \phi_t)^n=e^{\partial_t \phi+\phi_t} dV_Y
$$
on $\R^+ \times Y$ for some admissible volume form $dV_Y$.

Theorem \ref{thm:nkrf} gives a unique viscosity solution to this 
complex Monge-Amp\`ere flow with initial data 
$\phi_0  \in PSH(X,\chi_0) \cap {\mathcal C}^0(X)$.
This shows in particular that the K\"ahler-Ricci flow can be run on $Y$ from an 
initial data $S_0$  which is an arbitrary positive current in $c_1(K_Y)$
with continuous potentials.

\smallskip

It follows from \cite[Theorem 7.8]{EGZ09} that $Y$ admits a unique singular 
K\"ahler-Einstein current $S_{KE} \in c_1(K_Y)$, which is a smooth
bona fide K\"ahler-Einstein metric on the regular part $Y_{reg}$ of $Y$, and admits
globally continuous potentials at singular points $Y_{sing}$ \cite{EGZ11}.

\begin{theo} \label{eq:cvnkrf}
Given any initial data $S_0$  which is an arbitrary positive current with continuous potentials
in $c_1(K_Y)$, the normalized K\"ahler-Ricci flow 
$$
\frac{\partial \omega_t}{\partial t}=-\rm{Ric}(\omega_t)-\omega_t
$$
can be run from $S_0$ and converges, as $t \rightarrow +\infty$, towards $S_{KE}$.
\end{theo}

The convergence is uniform at the level of (properly normalized) potentials.
One can further show that the convergence holds in the ${\mathcal C}^{\infty}$-sense
in $Y_{reg}$ (see \cite{ST1}), if $S_0$ is a smooth K\"ahler form on $Y$.

\begin{proof}
We work on a log resolution $\pi:X \rightarrow Y$. Let $\theta_0 := \pi^* (\chi_0)$.
 Recall from \cite{EGZ09,EGZ11} that 
$$
\pi^*S_{KE}=\theta_0+dd^c \f_{KE},
$$
where $\f_{KE} \in PSH(X,\theta_0) \cap {\mathcal C}^0(X)$ is a viscosity/pluripotential solution
of the elliptic degenerate complex Monge-Amp\`ere equation
$$
(\theta_0+dd^c \f_{KE})^n=e^{\f_{KE}} \mu_{NKRF}.
$$
Thus $\f_{KE}$ is a fixed point (= static solution) of the NKRF and the comparison principle yields
$$
\| \f_t-\f_{KE} \|_{L^{\infty}(\R^+ \times X)} \leq \| \f_0-\f_{KE} \|_{L^{\infty}(X)}. 
$$

We can actually reinforce this uniform control by applying
 the comparison principle to the functions $u (t,x) =e^t \f(t,x)$ and $u_{KE} (t,x) =e^t \f_{KE} (x)$
which are $e^t \theta_0$-psh in $X$: observe indeed that $t \mapsto e^t \theta_0$ is non decreasing
and the $u_t$'s satisfy the twisted parabolic Monge-Amp\`ere equation
$$
(e^t \theta_0+dd^c u_t)^n=e^{e^{-t} \partial_t{u_t}+nt} \mu_{NKRF}.
$$
It follows therefore from Remark \ref{twist-PMA}
that for all $t>0$,
$$
\| \f_t-\f_{KE} \|_{L^{\infty}( X)} \leq e^{-t} \| \f_0-\f_{KE} \|_{L^{\infty}(X)},
$$
from which the conclusion follows.
\end{proof}

\subsubsection{Starting from an arbitrary class} Here we come back to the general case when $K_Y$ is nef and big.

\begin{theo}
Given any initial data $S_0$ which is an arbitrary positive current with continuous potentials in the K\"ahler class $\{\chi_0\}$, the K\"ahler-Ricci flow 
\begin{equation} \label{eq:NKRF}
\frac{\partial \omega_t}{\partial t}=-\rm{Ric}(\omega_t)-\omega_t
\end{equation}
can be run from $S_0$ and converges, as $t \rightarrow +\infty$, towards $S_{KE}$.
\end{theo}
Again the convergence is uniform at the level of (properly normalized) potentials.
One can further show that the convergence holds in the ${\mathcal C}^{\infty}$-sense
in $Y_{reg}$ (see \cite{ST1}), if $S_0$ is a smooth K\"ahler form on $Y$.

\begin{proof}  
Theorem~\ref{thm:nkrf} implies that  the equation (\ref{eq:NKRF}) has a unique solution starting from $S_0$.
It is clear that at the level of cohomology classes $\{ \omega_t\} \to c_1 (K_Y)$ as $t \to + \infty$.
We want to show that this is the case for the flow itself. This can be done using the comparison principle at the level of potentials.

We work on a log resolution $\pi:X \rightarrow Y$ so that (\ref{eq:NKRF}) is equivalent to the following Monge-Amp\`ere flow:
$$
(\theta_t+dd^c \f_t)^n=e^{\partial_t \f+\f_t} \mu_{NKRF},
$$
 where $\theta_t := \pi^* (\chi_t)$ and $\mu_{NKRF}$ is a volume form with canonical vanishing i.e. locally $\mu_{NKRF} =\Pi_E |f_E|^{2a_E} dV_X$.
 
 Let $\phi_{KE}$ be the potential of the singular K\"ahler-Einstein metric $S_{KE}$ on $Y$ given by \cite{EGZ09} i.e. $S_{KE} = \chi + dd^c \phi_{KE}$ and $(\chi + dd^c \phi_{KE})^n = e^{\phi_{KE}} d V_Y$.
 Define $\theta_{\infty} := \pi^*(\chi)$ and $\f_{KE} := \pi^* (\phi_{KE})$. Then the K\"ahler-Einstein equation can be written as 
 $$
 (\theta_{\infty}  + dd^c \f_{KE})^n = e^{\f_{KE}}  \mu_{NKRF}.
 $$

 The proof will be completed in three steps.
 
 \smallskip
\noindent {\bf Step 1:} We first establish a lower bound for the solution $ \f$
by finding an appropriate subsolution to the Cauchy problem for the flow (\ref{eq:PMAF}). Consider
$$
u (t,x) := e^{- t} \f_0 + (1-e^{- t}) \f_{KE} + h (t)e^{- t},
$$
on $\R^+ \times Y$,  where $h$ is a $C^1$ function in $\R^+$ to be chosen so that $u$ is a subsolution to the Cauchy problem for the flow (\ref{eq:PMAF}).

Observe that $ u (0,x) = \f_0$ if $h (0) = 0$ and for all 
$t > 0$, 
$$
\theta_t + dd^c u_t = e^{- t} (\theta_0 + dd^c \f_0) + 
(1-e^{- t}) (\theta_{\infty} + dd^c \f_{KE}) \geq 0
$$ 
in the weak sense of currents, hence  $u_t $ is $\theta_t$-psh and satisfies the inequality
$$
(\theta_t + dd^c u_t)^n \geq (1-e^{- t})^n (\theta_{\infty} + dd^c \f_{KE})^n = (1-e^{- t})^n e^{\f_{KE}} d V_Y.
$$
in the pluripotential sense on $X$. 

On the other hand $\partial_t u + u = \f_{KE} + h'(t) e^{- t}$  thus $u$ is a subsolution if 
$ (1-e^{- t})^n \leq e^{h'(t) e^{- t}}$. We therefore choose $h$ to be the unique solution of the ODE $h'(t)  =  n e^t \log (1-e^{- t})$ with $h (0) = 0$. We let the reader check that
$$
h(t)=n \left\{ (e^t-1)\log (e^t-1)-e^t \log (e^t) \right\}
=O (t)
\text{ as } t \to + \infty.
$$
 
It follows therefore  from Lemma  \ref{lem:pluripotvisc} that
 $u$  is a subsolution to the Cauchy problem for the normalized Monge-Amp\`ere flow (\ref{eq:PMAF}). By the comparison principle  we have $u\leq \phi$ in $ \R^+ \times X$ i.e. 
 \begin{equation} \label{eq:lb}
  \f_{KE} (x) - \f (t,x) \leq h (t) e^{-t} = O (t e^{-t}),
  \end{equation}
  for all $(t,x) \in \R^+ \times X$.

 \medskip

 The proof of the upper bound is  done by constructing an appropriate supersolution to the 
Cauchy problem. The construction is more involved and uses our earlier results in the degenerate elliptic case. We proceed in two steps

\smallskip
  
 \noindent {\bf Step 2:} We first assume that $K_Y$ is ample.
Fix $\beta$ an arbitrary K\"ahler form on $X$ and set $\theta_t := e^{-t} \theta_0 + (1 - e^{- t}) \beta$.  
Let $\f$ be the solution to the Monge-Amp\`ere flow 
 \begin{equation} \label{eq:PMAFbis}
 (\theta_t + dd^c \f_t)^n = e^{\partial_t \f + \f} \mu_{NKRF},
\end{equation}
 and let $ \p$ be the solution to the degenerate elliptic equation 
 \begin{equation} \label{eq:beta}
 (\beta + dd^c \p)^n = e^{\p} \mu_{NKRF}.
\end{equation}
  
Assume moreover that $\theta_0 \leq \beta$ and consider the function
$$
 v (t,x) := \p + C e^{- t},
$$
defined on $\R^+ \times X$, where $C := \max_X (\phi_0 - \p) > 0$ is chosen so that 
$v_0 = C + \p \geq \phi_0$ in $X$. 
This implies that $dd^c v_t + \theta_t \leq dd^c \p + \beta$ hence for all $t > 0$, 
$$
(dd^c v_t + \theta_t)^n \leq ( dd^c \p + \beta)^n = e^{\p} = e^{\partial_t v + v} \mu_{NKRF},
$$
in the sense of viscosity on $X$. 

Therefore $v$ is a supersolution to the flow (\ref{eq:PMAFbis}) and the comparison principle
yields the upper bound
$$
 \f (t,x) \leq \p (x)  + e^{- t} \max_X (\f_0 - \p)
$$

When $\beta$ is an arbitrary K\"ahler form on $X$, it follows from the definition of $\theta_t$ that there exists $T >> 1$ such that $\theta_t \leq 2 \beta$ for $t \geq T$. The K\"ahler-Ricci flow  starting from the current $\theta_T + dd^ c \f_T$ has a unique solution given by $ \phi (t,x) := \f (t + T,x)$ for $(t,x) \in \R^+ \times X$. Translating in time we can thus assume that $\theta_0 \leq 2 \beta$. Set
$$
v (t,x) := ( 1 + e^{-t}) \p (x) + h (t) e^{- t} + B e^{-t},
$$
where $h$ is a smooth function, $ h (0) = 0$ and  $B := \max_X (\f_0 - 2 \p)$ so that $v_0 \geq \f_0$. 
We want $v$ to be a supersolution of the flow (\ref{eq:PMAFbis}). 
Since $dd^c v_t + \theta_t \leq (1 + e^{-t}) (dd^c \p + \beta)$ we get
$$
(dd^c v_t + \theta_t)^n \leq (1 + e^{-t})^n (dd^c \p + \beta)^n = (1 + e^{-t})^n e^{\p}.
$$
Since $\partial_t v + v  = \p - h' (t) e^{- t}$ we impose $- h' (t) e^ {-t} = n \log (1 +  e^{-t})$. 
Observe again that $h (t) = O (t)$. By the comparison principle we conclude that $\f (t,x) \leq v (t,x)$ hence
\begin{equation} \label{eq:ub1}
  \f (t,x) \leq \p +  (\max_X (\f_0 - 2 \p) + \max_X \p) e^{- t} +   h (t)  e^{-t}.
 \end{equation}
 From (\ref{eq:lb}) and (\ref{eq:ub1}) we conclude, when $K_Y$ is ample
that $\vert \f_t - \f_{KE} \vert= O (t e^ {- t})$ as $t \to + \infty$.

\medskip

\noindent {\bf Step 3:} We now establish the upper bound when $K_Y$ is merely nef and big.
We set $\beta = \theta_{\infty} := \pi^* (\chi),$ where $\chi$ is semi-positive and big and represents the canonical class $K_Y$. The solution to the corresponding (\ref{eq:beta}) is the function function $\p = \f_{KE}$. 

We approximate $\beta$ by K\"ahler forms $\beta_\e := \beta + \e \eta$ for $\e > $ small enough, where $\eta > 0$ is a fixed K\"ahler form on $X$.
Set $\theta^\e _t:= e^{-t} \theta_0 + (1-e^{- t}) \beta_\e$ and solve as in Step 2 the corresponding complex Monge-Amp\`ere flow
\begin{equation} \label{eq:PMAFe}
(\theta^\e_t  + dd^c \f^\e_t)^n=e^{\partial_t \f^\e+\f^\e_t} \mu_{NKRF},
\end{equation}
with Cauchy data $\f^\e_0 = \f_0$ which is $\theta^\e_0$-psh in $X$ since $\theta^\e_0 = \theta_0$.
Let $\p^\e$ be the  continous $\beta_\e$-psh solution of the degenerate elliptic equation
$$
(\beta_\e + dd^c \p^\e)^n = e^{\p^\e} \mu_{NKRF},
$$
which exists by \cite{EGZ09}.
It follows from Step 2  that there exists $t_\e > 1$ such that for $t \geq t_\e$ and $x \in X$,
$$
 \f^\e (t,x) \leq  \p^\e (x) + e^{- t} \max_X (\f(t_\e,x) - 2 \p^\e(x)) +  h (t) e^{-t},
$$
where $h$ is a smooth function satisfying the $h' (t) e^ {-t} = n \log (1 + 2 e^{-t})$ with $h (0) = 0$.

Since $\theta \leq \theta^\e$, the function $\f$ is a supersolution to the parabolic equation (\ref{eq:PMAFe}) with the same Cauchy condition.
Moreover the family $t \longmapsto \theta^\e_t$ is very regular in the sense of Definition \ref{defi:oregular2}.
The comparison principle yields $ \f \leq \f^\e$ on $\R^+ \times X$.
Therefore  
\begin{eqnarray}
\label{eq:ub2}
 \f (t,x) - \f_{KE}(x)  &\leq & \p^\e (x) - \f_{KE}(x) \\
 \nonumber & + & \max_X (\f(t_\e,x) - 2 \p^\e(x)) \\
\nonumber & +& \max_X \p^\e + h (t)) e^{-t}.
\end{eqnarray} 
for $t \geq t_\e$ and $x \in X$.
The comparison principle shows that the family $(\p_\e)_{\e > 0}$ is non increasing and   $\p^\e \to \f_{KE}$ pointwise in $X$  as $\e \to 0$ (see \cite{EGZ11}).
The convergence $\p^\e \to \f_{KE}$ is uniform on $X$, as follows from Dini's lemma.

By using (\ref{eq:lb}) and (\ref{eq:ub2}), we conclude that $\f_t \to \f_{KE}$ uniformly on $X$ as $t \to + \infty$. Thus $\theta_t + dd^c \f_t \to \theta_{\infty} + dd^c \f_{KE}$. Pushing down to $Y$ we conclude that $\omega_t \to S_{KE}$ weakly on $Y$.
\end{proof}

\subsection{Calabi-Yau varieties}

Let $Y$ be a $\Q$-Calabi-Yau variety, i.e. a Gorenstein K\"ahler space of finite index with trivial first Chern class (see \cite[Definition 7.4]{EGZ09}). 

Fix $\chi_0$ a K\"ahler form on $Y$ and $S_0=\chi_0+dd^c \phi_0$ a positive closed current  with a continuous potential 
$\phi_0 \in PSH(Y,\chi_0) \cap {\mathcal C}^0(Y)$. The  K\"ahler-Ricci flow
$$
\frac{\partial \omega_t}{\partial t}=-\rm{Ric}(\omega_t)
$$
preserves the cohomology class $\{ \chi_0 \}$ since $c_1(Y)=0$. Thus
$\omega_t=\chi_0+dd^c \phi_t$ and the KRF 
can be written at the level of potentials as the complex Monge-Amp\`ere flow
$$
(\chi_0 +dd^c \phi_t)^n=e^{\partial_t \phi} dV_Y
$$
for some admissible volume form $dV_Y$

It follows  from Theorem \ref{thm:can} that the corresponding
complex Monge-Amp\`ere flow on a log resolution $\pi : X \longrightarrow Y$ with initial data $\f_0 := \phi_0 \circ \pi$ has a unique viscosity solution $\f$.
This shows in particular that the K\"ahler-Ricci flow in the sense of Definition \ref{krf} can be run on $Y$ from an 
initial data $S_0$  which is an arbitrary positive current with continuous potentials. 
The solution exists for all times $t>0$. Again, we recover
one of the main results of \cite{ST2}.   

\smallskip

It follows from \cite[Theorem 7.5]{EGZ09} that $Y$ admits a unique singular Ricci flat 
K\"ahler-Einstein current $S_{KE}$ in the K\"ahler class $\{\theta_0\}$, which is a smooth
bona fide K\"ahler-Einstein metric on the regular part $Y_{reg}$ of $Y$, and admits
globally continuous potentials at singular points $Y_{sing}$, thanks to \cite{EGZ11}.

\begin{theo}
Let $Y$ be a $\Q$-Calabi-Yau variety and fix $\al_0 \in {\mathcal K}(Y)$ a K\"ahler class.
Given any initial data $S_0 \in \al_0$  which is an arbitrary positive current with continuous potentials on $Y$,
the K\"ahler-Ricci flow 
\begin{equation} \label{eq:unkrf}
\frac{\partial \omega_t}{\partial t}=-\rm{Ric}(\omega_t)
\end{equation}
can be run from $S_0$ and converges, as $t \rightarrow +\infty$, towards the singular Ricci flat 
K\"ahler-Einstein current $S_{KE} \in \al_0$.
\end{theo}

The convergence here is uniform on $Y$ at the level of (properly normalized) potentials. A parabolic version
of Yau's ${\mathcal C}^2$-estimate, together with Tsuji's trick and 
parabolic Evan's-Krylov+Schauder theory allow to show that the convergence 
holds in the ${\mathcal C}^{\infty}$-sense
in $Y_{reg}$ (see \cite{ST1}) when $S_0$ is a smooth K\"ahler form on $Y$.

\begin{proof}
The  K\"ahler-Ricci flow (\ref{eq:unkrf}) is equivalent to the following  complex Monge-Amp\`ere flow on 
$X$, a log resolution $\pi : X \longrightarrow Y$
\begin{equation} \label{eq:unMAF}
(\theta_0 +dd^c \f_t)^n=e^{\partial_t \f} \mu_{NKRF},
\end{equation}
starting at $\f_0$ with the usual notations.

By Theorem~\ref{thm:can}, this flow has a unique solution $\f$ defined in $\R^ + \times X$. Observe that the solution $\f$ is uniformly bounded in $\R^+ \times X$. Indeed let $\rho$ be a solution to the degenerate elliptic equation $(\theta_0 + dd^c \rho)^n = d V_Y$ on $Y$ normalized by $\max_X (\f_0 - \rho) = 0$, which exists by \cite{EGZ09}.
The function $\p (t,x) := \rho(x)$ is a solution to the Monge-Amp\`ere flow (\ref{eq:unMAF}) with Cauchy condition $\p_0 = \rho$.
By the comparison principle we conclude that for any $(t,x) \in \R^+ \times X$, we have 
$$
  \rho (x) - \max_X (\rho - \f_0) \leq \f (t,x)  \leq  \rho (x).
$$
This shows that there exists uniform constants $m_0 , M_0$ such that $m_0 \leq \f (t,x) \leq M_0$ for all $(t,x) \in \R^ + \times X$.

The proof of the convergence theorem goes by approximating by  perturbed complex Monge-Amp\`ere  flows and by using the comparison principle as in the proof of \cite[Theorem 5.2]{EGZ14}.

\smallskip

We first prove an upper bound.
Consider the  flows
\begin{equation} \label{eq:MAFe}
(\theta_0+dd^c \phi_t)^n=e^{\partial_t \phi+\e (\phi - M_0)} d V_Y,
\end{equation}
starting at $\f_0$, where $\e>0$ is a parameter that we shall eventually let converge to zero.

By Theorem~\ref{thm:can}, the flow (\ref{eq:MAFe}) has a unique viscosity solution 
$\f^\e$ on $\R^+ \times X$.
Observe that $\f$ is a subsolution to this flow by the choice of $M_0$. 
The comparison principle thus insures
$$
\f (t,x) \leq \f^\e (t,x), \, \, \text{in} \, \, \, \R^+ \times X.
$$
It remains to estimate $\f^\e$ from above. 
For $\e>0$ fixed, the solution of the perturbed flow uniformly converges, as $t \rightarrow +\infty$,
to the solution of the static equation
$$
(\theta_0+dd^c u^\e)^n=e^{\e (u^\e - M_0)}  d V_Y,
$$
using a similar reasoning as in the previous section.

By the strong version of the comparison principle for the equation (\ref{eq:MAFe}) as in the proof of Theorem~\ref{eq:cvnkrf}, we have
$$
\max_{\R^ + \times X} \vert \phi^\e (t,x) - u^\e (x) \leq e^ {- \e t} \max_X \vert \f_0 (x) - u^\e (x)\vert.
$$

Moreover by stability of solutions to degenerate complex Monge-Amp\`ere equations established in \cite{EGZ11} we know that $u^\e \to u$ uniformly on $X$ to the solution $u$ of the equation $(\theta_0+dd^c u)^n=  d V_Y,$ normalized by the condition $\int_Y u d V_Y = 0$. We infer
\begin{equation} \label{eq:ubd}
\f (t,x) - u (x) \leq e^{- \e t}  \max_X \vert \f_0 (x) - u^\e (x)\vert + \max_X \vert u^\e (x) - u (x)\vert.
\end{equation}

\smallskip

We now take care of the lower bound. Consider for $\e > 0$
\begin{equation} \label{eq:MAFe}
(\theta_0+dd^c \psi_t)^n=e^{\partial_t \psi+\e (\psi - m_0)} d V_Y,
\end{equation}
starting at $\f_0$.
Observe that $\f$ is a supersolution to this flow by the choice of $m_0$.
Theorem~\ref{thm:can} guarantees that this flow has a unique viscosity solution $\psi^\e$. The comparison principle thus yields
$$
\psi^\e(t,x)  \leq \f (t,x), \, \, \text{in} \, \, \, \R^+ \times X.
$$

We now estimate $\p^\e$ from below. 
For $\e>0$ fixed, the solution of the perturbed flow uniformly converges, as $t \rightarrow +\infty$,
to the solution of the static equation
$$
(\theta_0+dd^c v^\e)^n=e^{\e (v^\e} - m_0)  d V_Y,
$$
Again by stability of solutions to degenerate complex Monge-Amp\`ere equations established 
in \cite{EGZ11} we know that $v^\e \to u$ uniformly on $X$, where $u$ is the unique solution  of the equation $(\theta_0+dd^c u)^n=  d V_Y,$ normalized by the condition $\int_X u \mu_{NKRF} = 0$.
As above we obtain the lower bound 
\begin{equation} \label{eq:lbd}
u (x) - \f (t,x) \leq \e^{- \e t}  \max_X \vert v^\e (x) - \f_0 (x)\vert + \max_X \vert u (x) - v^\e (x)\vert.
\end{equation}
It is now clear from (\ref{eq:ubd}) and (\ref{eq:lbd}) that $\f_t \to u$ uniformly in $X$ as $t \to + \infty$.

 Pushing down everything to $Y$ we see that  $\omega_t =\theta_0 + dd^c \phi_t \to \theta_0 + dd^c u = S_{KE}$, as $t \to + \infty$,  as claimed.
\end{proof}

\subsection{Smoothing properties of the K\"ahler-Ricci flow}

Smoothing properties of the K\"ahler-Ricci flow have been observed and used by many authors in the last thirty years (see e.g. \cite{BM87,Tian97,PSSW08}). 

Attempts to run the K\"ahler-Ricci flow from a degenerate initial data have
motivated several recent works \cite{CD07,CT08,CTZ11,ST2,SzTo}. 
The best result (before \cite{GZ13}) is that of Song and Tian \cite{ST2}
who showed that  on a projective variety $Y$ with canonical singularities,
the K\"ahler-Ricci flow 
$$
\frac{\partial \omega_t}{\partial t}=-\rm{Ric}(\omega_t)
$$
can be run from an initial data $T_0=\chi_0+dd^c \phi_0$ which is a positive current
with continuous potentials.\footnote{The precise assumption in \cite{ST2} is a bit more restrictive but
can easily be extended to this statement as observed in \cite{BG13}.}
It is then classical that the flow exists on a maximal interval of time $[0,T_{max}[$, where
$$
T_{max}=\sup \{t>0 \, | \, \{ \omega_0\}-tc_1(Y) \text{ is K\"ahler } \}.
$$

The parabolic viscosity approach we have developed in this article allows us to show that
the potentials constructed in all these works are globally continuous on $ [0,T_{max}[ \times Y$.

\begin{theo}
Let $Y$ be a projective variety with at worst canonical singularities. Fix 
$\chi$ a smooth closed form representing $c_1(K_Y)$, $\chi_0$ a K\"ahler form on $Y$
and let  $S_0=\chi_0+dd^c \phi_0$ be a positive current with a continuous potential on $Y$.
The K\"ahler-Ricci flow  with initial data $S_0$
$$
\frac{\partial \omega_t}{\partial t}=-\rm{Ric}(\omega_t)
$$
 admits a unique solution $\omega_t=\chi_0 + t\chi+ dd^c \phi_t$, with 
\begin{itemize}
\item for all $0<t<T_{max}$, the function $x \mapsto \f_t(x)$ is a $\chi_t$-psh function on $Y$
which is smooth in $Y_{reg}$;
\item $(t,x) \mapsto \f(t,x)$ continuous on $[0,T_{max}[ \times Y$.
\end{itemize}
\end{theo}

\begin{proof} 
When $K_Y$ is semi-ample, we can assume $\chi \geq 0$ hence 
$t \mapsto \theta_t =\theta_0 +t\chi$ is non-decreasing.
In the general case since $\theta_0 > 0$ is K\"ahler, there exists a constant $A > 0$ such that $- \chi \leq  A \theta_0$ in $Y$.
Therefore the family $t \longmapsto \theta_t =\theta_0 +t\chi$ is very regular in the sense of Definition \ref{defi:oregular2}.
The result is thus  an immediate consequence of Theorem \ref{thm:can}.
\end{proof}

The continuity of $\f$ at singular points of $Y_{sing}$ is the novelty here:
this is the parabolic analogue of the main application of \cite{EGZ11}.

 For complex Monge-Amp\`ere flows starting from even more degenerate initial data, we refer
the reader to \cite{GZ13}, where the work of Song-Tian is extended so as to allow the  
K\"ahler-Ricci flow to be run from a positive current with zero Lelong numbers.
Our viscosity approach can also be used in this latter context to show that
the maximal solution of  the K\"ahler-Ricci flow becomes immediately smooth on $Y_{reg}$,
for $t>0$,
with globally continuous potentials on $Y$.

\section{Concluding remarks: the K\"ahler-Ricci flow over flips} \label{overflip}

The extinction time of the KRF on $Y$ can be expressed as 

$$
T_0= \sup\{ t>0, \{\omega_0\}+t K_Y\} \in \mathcal{K}(Y).
$$
Let us assume that $(Y,\{\omega_0\})$ satisfies the following assumptions: 
\begin{itemize}
 \item 
 $Y$ has terminal singularities. 
\item 
$T_0<\infty$. 
\item 
$\{\omega_{T_0}\}=\{\omega_0 \} +T_0 K_Y$ 
is {\em a non trivial pull back from a K\"ahler class}, i.e.: that there exists a non-biholomorphic proper bimeromorphic
holomorphic map $\psi^{-}: Y\to Z$  such that $Z$ is a normal K\"ahler complex space
 and
$\{\omega_0\}+T_0 K_Y \in (\psi^{-})^*\mathcal{K}(Z)$.
\item For $N\in \N^*$ divisible enough the sheaf of graded algebras  $$\mathcal{P}(Y/Z):=\bigoplus_{n\in \N}\psi^{-}_* O_Y(nN K_Y)$$
is locally finitely generated over $O_Z$. 
\end{itemize}

The last condition is fulfilled thanks to \cite[Thm 1.2 (3)]{BCHM} if $Y$ and $Z$ are projective varieties. 
We then denote by $\psi ^+: Y^+\to Z $ the relative canonical model 
of  $\psi^{-}: Y \to Z$, namely $Y^+:=\mathrm{Proj}( \mathcal{P}(Y/Z))$ . 
It is known thanks to the classical work of M. Reid that $Y^+$ is normal (and has canonical singularities)
and it is trivial to see that $\psi^+$ is a proper bimeromorphic mapping.

It follows from \cite[Lemma 3.38]{KM}\footnote{Stated for algebraic varieties. 
The proof however goes through in the complex analytic category
since $\psi^{-}$ is a projective morphism due to the fact that $-K_Y$ is $\psi^{-}$-ample.} that $Y^+$ has terminal 
singularities. Also, if $Y$, $Z$
are projective and $Y$ is  $\Q$-factorial, then $Y^+$ is $\Q$-factorial.
 One can  construct a diagram:
$$
\xymatrix@1{
& X \ar[dl]_{\pi^-} \ar[dr]^{\pi^+} & \\
Y\ar[dr]_{\psi^-} & & Y^+\ar[dl]^{\psi^+} \\
& Z & }
$$
where $X$ is smooth,  $\pi^{-}$, $\pi^{+}$ are log-resolutions
such that $\mathrm{Exc}(\pi^+)\cup \mathrm{Exc}(\pi^+)$ is a divisor with simple normal crossings,
 $\psi^+$ $\psi^{-}$ are proper bimeromorphic holomorphic map.  By construction, 
$-K_{Y}$ is $\psi^{-}$-ample, $K_{Y^+}$ is $\psi^+$-ample
and one has the following properties:
\begin{lem}

There exists a real number $\epsilon>0$ such that for $t\in ]T_0,  T_0+\epsilon[$,
$$
 \{\omega_{T_0}\}+ (t-T_0) K_{Y^{+}} \in \mathcal{K}(Y^{+}).
$$ 
 \end{lem}

 \begin{proof}
Immediate consequence of the fact that $K_{Y^+}$ is $\psi^+$-ample.
\end{proof}

\begin{lem}
 The exceptional divisors of $\pi^-$ are 
 exceptional for $\pi^+$. 
\end{lem}

\begin{proof}
The bimeromorphic map $\xymatrix{Y^{+} \ar@{-->}[rr]^{(\psi^-)^{-1} \circ \psi^+}&&  Y}$
contracts no divisor, since a log-canonical model is a contraction and $Y^+\to Z$ is the log-canonical model of $Y\to Z$ see \cite[section 3]{BCHM}. 
\end{proof}

Furthermore,  for an exceptional divisor $E$ of $\pi^+$,  we have $a_E\le a_E(Y^+):=a_E^+$
where $a_E=0$ if $E$ is not $\pi^-$-exceptional
by \cite[Lemma 3.38]{KM} and  we define $\delta_E=a^+_E-a_E \geq 0$. 

\medskip

We define a measurable
volume form with semi positive continuous density on $X$ by 
$$
\mu= \left(\mathbb{I}_{t\le T_0}+ \mathbb{I}_{\{t> T_0\}}(\prod_E  |s_E |_{h_E}^{2\delta_E}) \right)\mu_{NKRF}(h^-)
$$
and $\bar\omega \in H^0(X, \mathcal{Z}^{1,1}_{X_{T_0+\epsilon}/[0,T_0+\epsilon[})$ by 
$$ 
\bar \omega_t= \omega_0+\int_0^t du \ (dd^c\log(\mu) - a_E[E] -\mathbb{I}_{\{u> T_0\}} \delta_E [E]).
$$

The fact that $\bar \omega$ has continuous local potentials is straightforward.
The pair $(\bar \omega,\mu)$ defines a K\"ahler Ricci flow on $Y=Y^-$ for $t<T_0$ 
and a K\"ahler Ricci flow on $Y^+$ for $t>T_0$. On the other hand 
the flow $(CMAF)_{\bar\omega, V}$ does not satisfy condition (\ref{vreg}) at $T_0$. Indeed in every coordinate system
one can find a potential in such a way that this flow has the following expression:
$$ (dd^c \phi)^n=e^{\frac{\partial \phi}{\partial t}} |z_E|^{2a_E + 2\mathbb{I}_{\{t> T_0\}} \delta_E}.
$$

We believe a large part of the theory developed here should hold in spite of the breakdown of condition (\ref{vreg}) but we
shall not treat any further this topic in the present article and hope to return to that problem in a  later work.


\end{document}